\documentclass[a4paper,11pt,reqno]{article}
 \usepackage[english]{babel}
  \usepackage{amsmath,amsfonts,amssymb,amsthm,bbm}
   \usepackage{graphics,epsfig,psfrag} %%add this and next lines if pictures
   \usepackage{paralist,subfigure}
   \usepackage[dvipsnames]{xcolor}
   \usepackage{hyperref} %%Warning: when you first run your tex
    \usepackage[active]{srcltx} %%allows you to jump from your editor
    \usepackage{color,soul} %%allows you to use the \st{...} command, which
    \usepackage[utf8]{inputenc}
\usepackage[T1]{fontenc}
    \usepackage{authblk}
    \usepackage{esint}
    \usepackage{comment}
\usepackage{graphicx}
\usepackage{subcaption}
\usepackage{float}

\newtheorem{theorem}{Theorem}[section]
\newtheorem*{theorem*}{Theorem}
\newtheorem{corollary}[theorem]{Corollary}
\newtheorem{lemma}[theorem]{Lemma}
\newtheorem{prop}[theorem]{Proposition}
\newtheorem{conj}[theorem]{Conjecture}

\theoremstyle{definition}

\newtheorem{remark}{Remark}

\def\N{\mathbb{N}}
\def\Z{\mathbb{Z}}

\def\R{\mathbb{R}}

\let\e=\varepsilon

\let\ol=\overline
\let\ul=\underline
\let\.=\cdot
\let\0=\emptyset

\def\1{\mathbbm{1}}

\newenvironment{formula}[1]{\begin{equation}\label{#1}}
                       {\end{equation}\noindent}

\def\Fi#1{\begin{formula}{#1}}
\def\Ff{\end{formula}\noindent}

\title{\bf Propagation of opinions on a network. \\
Convergence, spreading speed and patterns in the Daley-Kendall model.}

\author{Romain \textsc{Ducasse}}
\affil{Université Paris Cité, CNRS, Sorbonne Université, Laboratoire Jacques-Louis Lions (LJLL), F-
75006 Paris, France}
\date{}
\begin{document}

\maketitle

\noindent {\textbf{Keywords:} Propagation of opinions, Propagation on graphs, Spreading speed, Daley-Kendall model, Patterns.} \\

\noindent {\textbf{MSC:} 34A34, 34C60, 91D25, 92D30, 05C82.}
\begin{abstract}
We study a social contagion model describing the propagation of opinions across a network. Our model is a generalization of a model introduced by Daley and Kendall \cite{DaleyKendall}, and consists of a system of coupled differential equations set on the nodes of an infinite graph.

The modeling is inspired by SIR-type models from epidemiology. The key idea is that individuals can transmit their opinion to their neighbors (as in epidemiological models), until they perceive that the opinion is already known in their neighborhood.

After introducing the model, we analyze the long-time behavior of its solutions. We prove that the solutions converge and we study their asymptotic equilibrium. We also establish estimates on the speed of propagation of the solutions through the network. 

These results give estimates on the rate of adoption of the opinion and on its velocity in the social network.

We also provide simulations that indicate that this model can give rise to patterns: when the individuals take into account individuals that are ``far enough'' in estimating  whether or not the opinion is already known, then we observe the emergence of localized ``bubbles'' where the opinion is adopted.
\end{abstract}

%-------------------------------------------------------------------------------

%\newpage

\section{Introduction}

\subsection{Propagation of opinions.}

The concept of {\em social contagion} states that sociocultural phenomena (such as ideas, opinions, rumors, fake news, etc) can spread through a population {\em similarly to} outbreaks of diseases. This concept was first developed by social psychologists in the late $19$th century, in particular in the works of Tarde \cite{tarde1890lois} or Baldwin \cite{Baldwin}, and found its first empirical confirmations starting in the $1950$s. We refer to \cite{MarsdenSocialContagion} and the references therein for more details on this topic.\\

On the other hand, since the pioneering works of Kermack and McKendrick \cite{KermackContributions91} in $1927$, who introduced the celebrated SIR model to describe the spread of a disease in a population, the mathematical study of epidemiological models has been widely developed.\\

This led several authors to adapt epidemiological models in order to build new models for the spread of sociocultural phenomena. For instance, \cite{BettencourtPower06} studies the spread of the adoption of the technique of Feynman diagrams by physicists, while \cite{Bonnasse-GahotEpidemiological18} studies the spread of riots in 2005 in France, and \cite{BassNew69} studies the adoption of a new product by consumers. All these works use almost directly the original SIR model of Kermack and McKendrick. 

In $1964$, Daley and Kendall introduced in \cite{DaleyKendall} a new model to describe the spread of rumors. Although their model is inspired by the SIR model, it includes a new effect, the {\em stifling}. In the Daley-Kendall model, the individuals can transmit their opinion (following the same procedure as in the SIR model) but cease to transmit when they think that the opinion is already sufficiently known (when it is not ``new enough''). This {\em stifling} process differs strongly from what can be found in standard epidemiological models (where the individuals can not actively choose to stop being infectious).\\

In the present paper, we study how opinions spread through a social network. To do so, we introduce and study a natural generalization of the Daley-Kendall model on a graph. The model we consider, stated as \eqref{syst} below, consists of a system of ODEs defined on the nodes of an infinite graph. We study the long-time behavior of the solutions. Doing so, we will obtain estimates on the fraction of population that eventually adopts the new opinion. We also show that the opinion propagates in a wave-like fashion and we give some estimates on its speed of propagation.

Finally, we provide some simulations that indicate that the stifling mechanism can lead, in some cases, to the emergence of {\em patterns}. From a sociological point of view, this could account for the emergence of ``bubbles'' of opinions.

\subsection{The Daley-Kendall model}

Before introducing our model, let us present here the original Daley-Kendall model, introduced in \cite{DaleyKendall}. The Daley-Kendall model is a compartmental model. The population is divided into three groups (compartments): the {\em ignorants} (who have not adopted the opinion), the {\em spreaders} (who adopted the opinion and are spreading it) and the {\em stiflers} (who adopted the opinion but stopped spreading it).

The opinion is propagated through the population by pair-wise contacts between spreaders and ignorants, following a law of mass-action. Any spreader involved in a pair-wise meeting with an ignorant can ``infect'' the ignorant with the opinion, turning them into a new spreader. This process is similar to the process of {\em contamination} in epidemiological models. 

On the other hand, when a spreader meets someone who is already aware of the opinion (that is, either another spreader or a stifler), then the spreader becomes aware that the opinion is ``known'' and can decide to stop spreading it, turning into a stifler. Such process does not exist in epidemiological models: usually, the infectious individuals (who are the equivalent of the spreaders) automatically recover after some time.\\

Denoting $S,I,R$ the number of ignorants, spreaders and stiflers respectively in the population\footnote{We use the letter $S,I,R$ as a reference to the epidemiological setting, where the compartment $S$ is the compartment of the susceptible, who have not been infected, $I$ is for the infectious and $R$ for the recovered. }, the original Daley-Kendall model takes the form of the following system of ODEs
\begin{equation}\label{DK ODE}
    \left\{
\begin{array}{rll}
    S'(t) &= -\alpha S(t)I(t),\quad &t>0, \\
      I'(t) &=  \alpha S(t)I(t) - \beta I(t)(I(t)+R(t)),\quad &t>0,\\
      R'(t) &=\beta I(t)(I(t)+R(t)),\quad &t>0.
\end{array}
    \right.
\end{equation}
Here, $\alpha>0$ is the rate of adoption of the opinion when an encounter between an ignorant and a spreader occurs, while $\beta>0$ is the rate at which a spreader ceases to spread the opinion when encountering other people already aware of the opinion.

This model has to be complemented with initial data
$$
S(t=0)=S_0,\quad I(t=0)=I_0,\quad R(t=0)=0.
$$
This initial data represent the initial state of the population. It is classical to assume that $R(t=0)=0$, that is, at the initial time, the opinion is ``new'', only some individuals have adopted it and there are no stiflers yet. The question is then to determine the conditions under which the opinion can spread to the rest of the population. Observe that the total population $P:=S(t)+I(t)+R(t)$ is constant over time, there is no death or birth. Therefore, the total population is always equal to $S_0+I_0$.\\

The main result of Daley and Kendall \cite{DaleyKendall} is that, no matter how small $I_0>0$, there is a proportion of the population, strictly positive but strictly less than $1$, that will eventually adopt the opinion. This proportion is ``almost'' independent of the initial population. More precisely, their result can be stated as follows:
 \begin{theorem*}[Daley-Kendall (1965)]
   Let $(S,I,R)$ be the solution of \eqref{DK ODE} arising from the initial datum $(S_0,I_0,0)$, where $S_0,I_0>0$.\\

  We have
   $$
   S(t)\underset{t\to+\infty}{\longrightarrow} S_0e^{-\alpha \theta},\quad I(t)\underset{t\to+\infty}{\longrightarrow}0,\quad R(t)\underset{t\to+\infty}{\longrightarrow} S_0(1-e^{-\alpha \theta}) +I_0,
   $$
   where $\theta>0$ is the unique positive solution of the equation
   $$
   \left(\frac{\alpha+\beta}{\alpha}\right)S_0(1-e^{-\alpha \theta}) + I_0= \beta(S_0+I_0)\theta.
   $$
 \end{theorem*}
Let us comment on this result. We can define the {\em rate of adoption} of the opinion as the proportion of individuals that have adopted the opinion (that is, the ratio of spreaders and stiflers over the total population)
$$
r(t) = \frac{I(t)+R(t)}{S_0+I_0} = 1 - \frac{S(t)}{S_0+I_0}.
$$
Then, we have that $r(t)\underset{t\to+\infty}{\longrightarrow} 1-\frac{S_0}{S_0+I_0}e^{-\alpha \theta} \in ]0,1[$. This means that there is a strictly positive proportion of the population that eventually adopts the opinion, and this proportion is strictly smaller than $1$, so the adoption is never total.

When $I_0$ goes to zero, then $\theta$ converges to $\theta_0$, the only strictly positive solution of the equation $\left(\frac{\alpha+\beta}{\alpha}\right)(1-e^{-\alpha \theta_0}) = \beta\theta_0$, which is independent of $S_0$, and the rate of adoption converges to $1 - e^{-\alpha \theta_0}$, which is then also independent of the size of the initial population.\\

In their paper, Daley and Kendall emphasize the differences between their result and the standard results concerning the usual SIR model from epidemiology. First, in the SIR model (and in most models in epidemiology), there is a {\em threshold} above which the epidemic does not propagate (more precisely, there is a quantity, called the {\em basic reproduction number}, that can be computed from $S_0$ and the parameters of the model, which needs to be strictly greater than $1$ for the disease to propagate, see \cite{murray2011mathematical} for instance). In the Daley-Kendall model, this is not the case, the opinion always propagate. We say that the Daley-Kendall model shows a {\em hair-trigger effect}: even a vanishingly small amount of ``infected'' will lead to a spread of the opinion in the population.

Daley and Kendall also emphasize that the proportion of the population that eventually adopts the opinion is independent of $S_0$ in the limit where $I_0$ goes to zero (it is equal to $e^{-\alpha \theta_0}$). This is also very different from what happens in epidemiology, where the proportion of infected individuals goes to zero when $S_0$ is very large (for instance, high densities area have usually more contamination {\em per capita} than less densely populated areas).

\subsection{The Daley-Kendall model on a network}\label{sec DK net}

We consider in this paper a natural extension of the Daley-Kendall model set on a network. More precisely, we consider a (countably infinite) family of communities. For notational simplicity, we assume  that they are indexed by the integers $n\in\N$. Each community has its own number of ignorant, spreaders and stiflers, denoted by $S_n(t),I_n(t),R_n(t)$ respectively. The key idea of the model is that the individuals in a community can be influenced not only by members of the community but also by members of neighbors communities.\\

We represent the interactions between different communities with a graph. Let $G$ be a countably infinite\footnote{All the analysis provided in this paper also applies to finite graphs, but the case of infinite graph is more interesting when studying the speed of propagation of the opinion.} graph whose nodes are indexed by the integers $\N$. The node $n$ will represent the community $n$. We shall denote $\sim$ the adjacency relation: $i\sim n$ means that the nodes $i$ and $n$ are neighbors. By convention, we also always assume that $n\sim n$ for all $n\in\N$.\\

For each couple $i,n\in\N$ such that $i\sim n$, we associate a weight $\alpha_{i,n}>0$ that will represent the strength of the influence of the community $i$ toward the community $n$. These need not be symmetric: a priori, $\alpha_{i,n}\neq \alpha_{n,i}$. For notational simplicity, we define $\alpha_{i,n}=0$ when $i\not\sim n$.

 In the epidemiological setting, $\alpha_{i,n}$ would be a {\em rate of contagion}.\\

We also consider a family of non-negative numbers $(\beta_{i,n})_{(i,n)\in \N^2}$, that represents the rate of {\em stifling} for individuals located at the position $n$, induced by observing the position $i$. We mention that $\beta_{i,n}$ could be strictly positive even if the nodes $i,n$ are not adjacent. This could reflect the fact that individuals can observe the presence of an opinion in the population without direct interaction, but through various communication channels (such as internet, media, etc).\\

Our model is the following system of ODEs:
\begin{equation}\label{syst or}
\left\{
\begin{array}{rll}
     S'_n(t) &= - S_n(t)\left(\sum_i \alpha_{i,n}I_i(t) \right),\quad &t>0,\ n\in \N,\\
    I'_n(t) &=  S_n(t)\left(\sum_i \alpha_{i,n}I_i(t) \right) - I_n(t)\left(\sum_i \beta_{i,n}(I_i(t)+R_i(t)) \right),\quad &t>0,\ n\in \N,\\
    R'_n(t) &=   I_n(t)\left(\sum_i \beta_{i,n}(I_i(t)+R_i(t)) \right),\quad &t>0,\ n\in \N,
\end{array}
\right.
\end{equation}
completed with initial data
$$
S_n(t=0)=S_n^0,\ I_n(t=0)=I_n^0,\ R_n(t=0)=0.
$$
Observe that the term $\sum_i \alpha_{i,n}I_i$ could be rewritten $\sum_{i\sim n} \alpha_{i,n}I_i$. This is not the case for the sum of the stifling terms.\\

Because we have, for all $n$, that $(S_n(t)+I_n(t)+R_n(t))'=0$, the total population at each node is constant equal to $S_n^0+I_n^0$. From now on, we call $P_n := S_n^0+I_n^0$ the total population at node $n$. We can rewrite the system \eqref{syst or} as
\begin{equation}\label{syst}
\left\{
\begin{array}{rll}
     S'_n(t) &= - S_n(t)\left(\sum_i \alpha_{i,n}I_i(t) \right),\quad &t>0,\ n\in \N,\\
    I'_n(t) &=  S_n(t)\left(\sum_i \alpha_{i,n}I_i(t) \right) - I_n(t)\left(\sum_i \beta_{i,n}(P_i - S_i(t)) \right),\quad &t>0,\ n\in \N.
\end{array}
\right.
\end{equation}
{\bf Local stifling}. We will pay a particular attention to the case where the stifling is {\em purely local}, that is, when the stifling parameters $(\beta_{i,n})_{(i,n)\in\N^2}$ are such that there are $\beta_n>0$ such that
\begin{equation}\label{beta local}
\beta_{i,n} = \begin{cases}
    \beta_n \quad \text{if} \ i=n,\\
    0 \quad \quad \text{else}. 
\end{cases}
\end{equation}
In this case, the model rewrites
\begin{equation}\label{syst loc}
\left\{
\begin{array}{rll}
    S_n'(t) &= - S_n(t)\left(\sum_i \alpha_{i,n}I_i(t) \right),\quad &t>0,\ n\in \N,\\
    I_n'(t) &=  S_n(t)\left(\sum_i \alpha_{i,n}I_i(t) \right) - \beta_n I_n(t)(P_n - S_n(t)) ,\quad &t>0,\ n\in \N.
\end{array}
\right.
\end{equation}
In the local stifling case, the individuals evaluate whether to cease to propagate the opinion depending only on how much it is known in their own community.\\

The goal of the present paper is to understand {\em how much} and {\em how fast} a new opinion is adopted in the network. To do so, it is convenient to consider the rate of adoption of the opinion at node $n$, 
\begin{equation}\label{rate}
r_n(t) := \frac{I_n(t)+R_n(t)}{P_n} = 1 - \frac{S_n(t)}{P_n}.
\end{equation}
Here, $r_n(t)\in [0,1]$ is the proportion of individuals that have adopted the opinion at time $t$ and at node $n$. When $r_n(t)=0$, the opinion is completely ignored by the community (at time $t$), while the larger $r_n(t)$, the more adopted the opinion is.\\

The questions that we discuss in this paper are the following:\\

{\bf Do we have propagation of the opinion ?} At the initial time, the opinion is present only in a finite number of nodes: we have $r_n(0)=0$ except for a finite amount of nodes. Does the opinion eventually (when $t\to+\infty$) reach all the nodes of the network ? In other words, denoting $r_n^\infty := \lim_{t\to +\infty} r_n(t)$, can we find $\ul r>0$, independent of $n$, such that $r_n^\infty \geq \ul r$ ? The answer is yes, we shall prove the convergence of $r_n(t)$ and give estimates on its limit values.\\

{\bf What is the speed of spreading of the opinion on the network ?} Once that we know that the opinion spreads, can we estimate the time it would take for an opinion to reach some node $n$ ? More precisely, if we know that $\lim_{t\to + \infty} r_n(t) \geq \ul r>0$, can we estimate the smallest $T_n>0$ for which we will have $r_n(T_n) \geq \frac{\ul r}{2}$ ? We will show that, for nodes $n$ far away from the initial focus of presence of the opinion, we will have $\frac{T_n}{n} \in [\frac{1}{c^\star},\frac{1}{c_\star}]$, where $c^\star,c_\star>0$ can be computed from the parameters of the model. This means that the opinion spreads at least with speed $c_\star$ and at most with speed $c^\star$.\\

To illustrate this, let us show a simulation of \eqref{syst loc} (our system with local stifling). We take $G$ to be a random graph (we take a Watts-Strogatz random graph, such graphs appear in the modeling of social networks, see \cite{watts1998collective}) with $50$ nodes. At time $t=0$, we have $I_0^0 = 0.1$ and $I_n^0=0$ for all $n\neq 0$ (the opinion is initially present only in the node $n=0$) and $S_n^0 = 1$. We plot the evolution of the system at $6$ different times. The image on the left represents the graph, the color on each node represents the rate of adoption $r_n(t)$. The image on the right shows the number of ignorants and spreaders at each node.

\begin{figure}[H]
\centering

% -------- Row 1 --------
\begin{minipage}{0.48\textwidth}
\centering
\includegraphics[width=\linewidth]{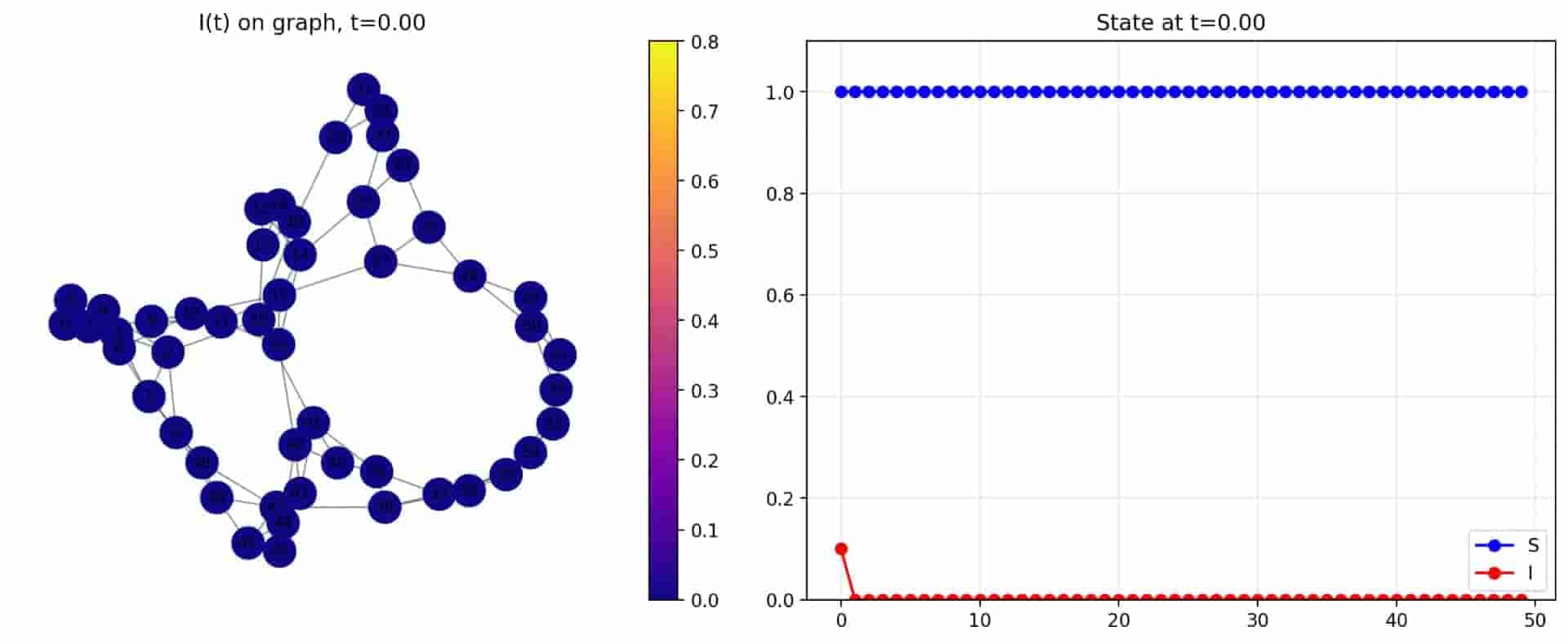}
\end{minipage}
\hfill
\begin{minipage}{0.48\textwidth}
\centering
\includegraphics[width=\linewidth]{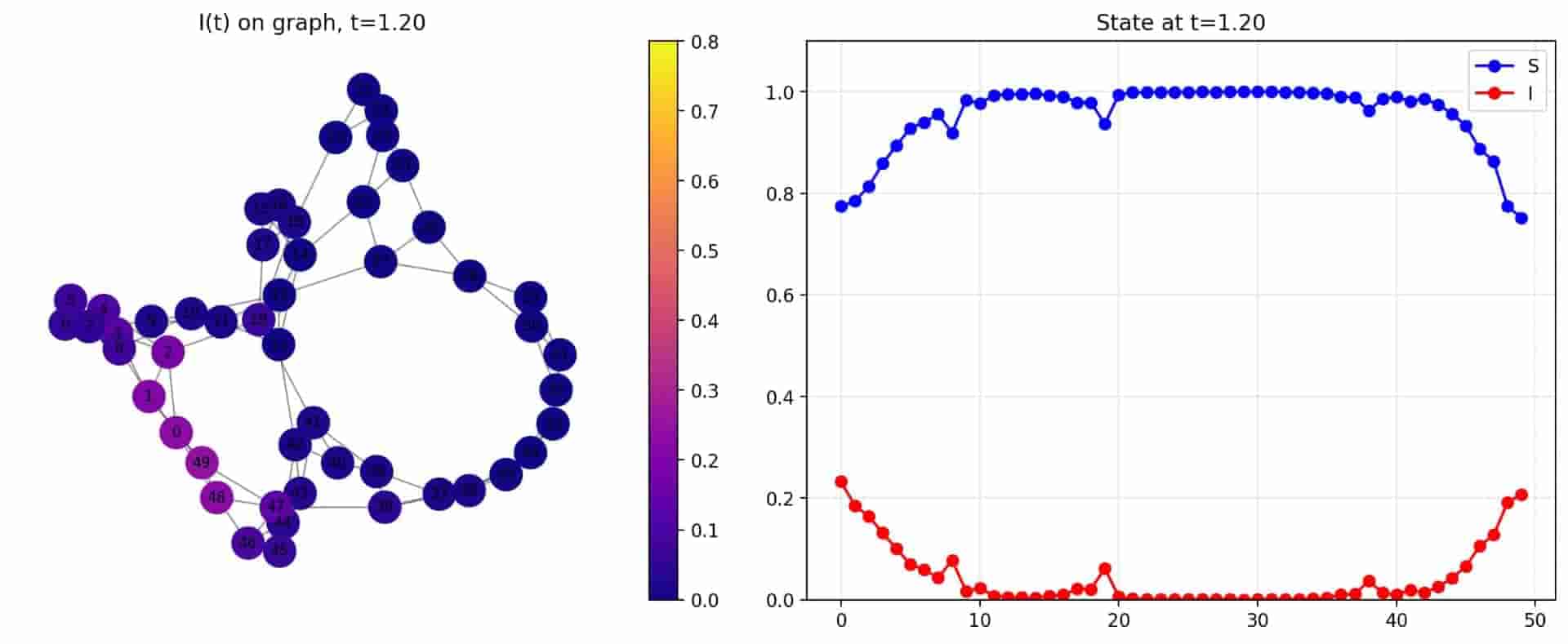}
\end{minipage}

\vspace{0.5cm}

% -------- Row 2 --------
\begin{minipage}{0.48\textwidth}
\centering
\includegraphics[width=\linewidth]{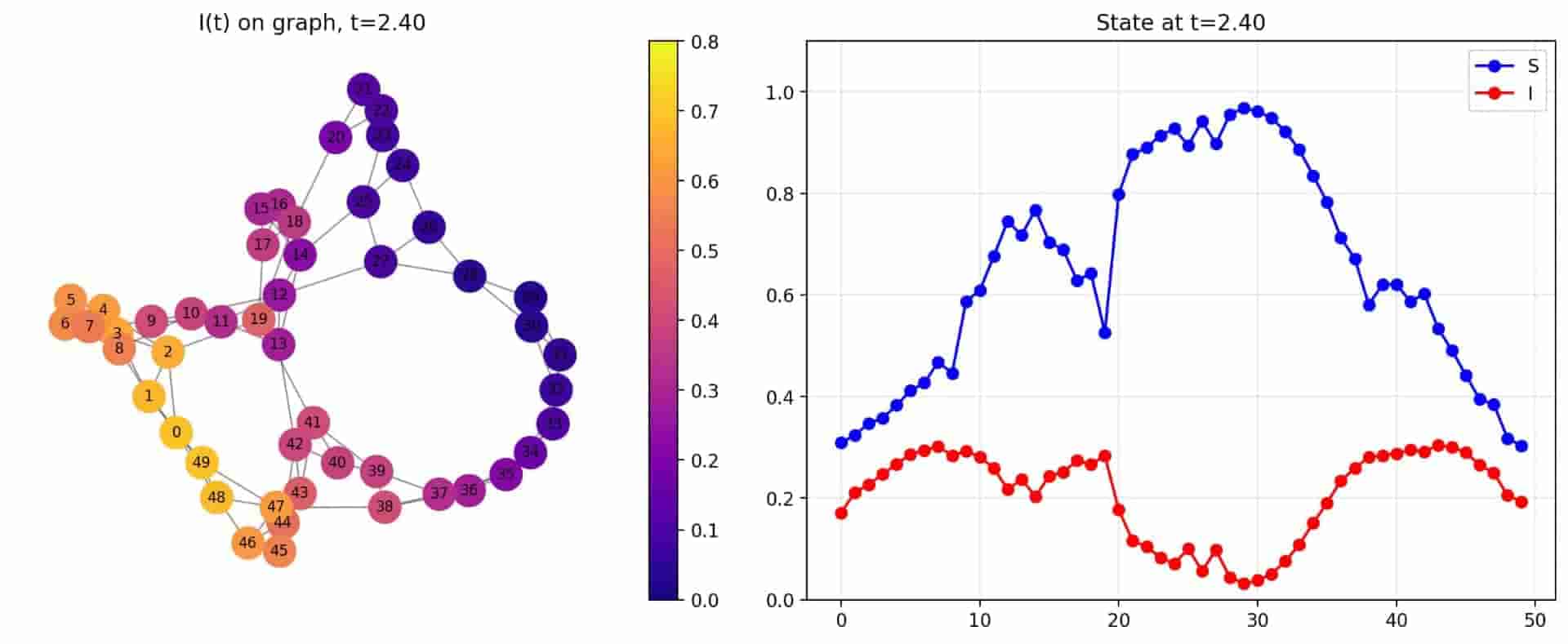}
\end{minipage}
\hfill
\begin{minipage}{0.48\textwidth}
\centering
\includegraphics[width=\linewidth]{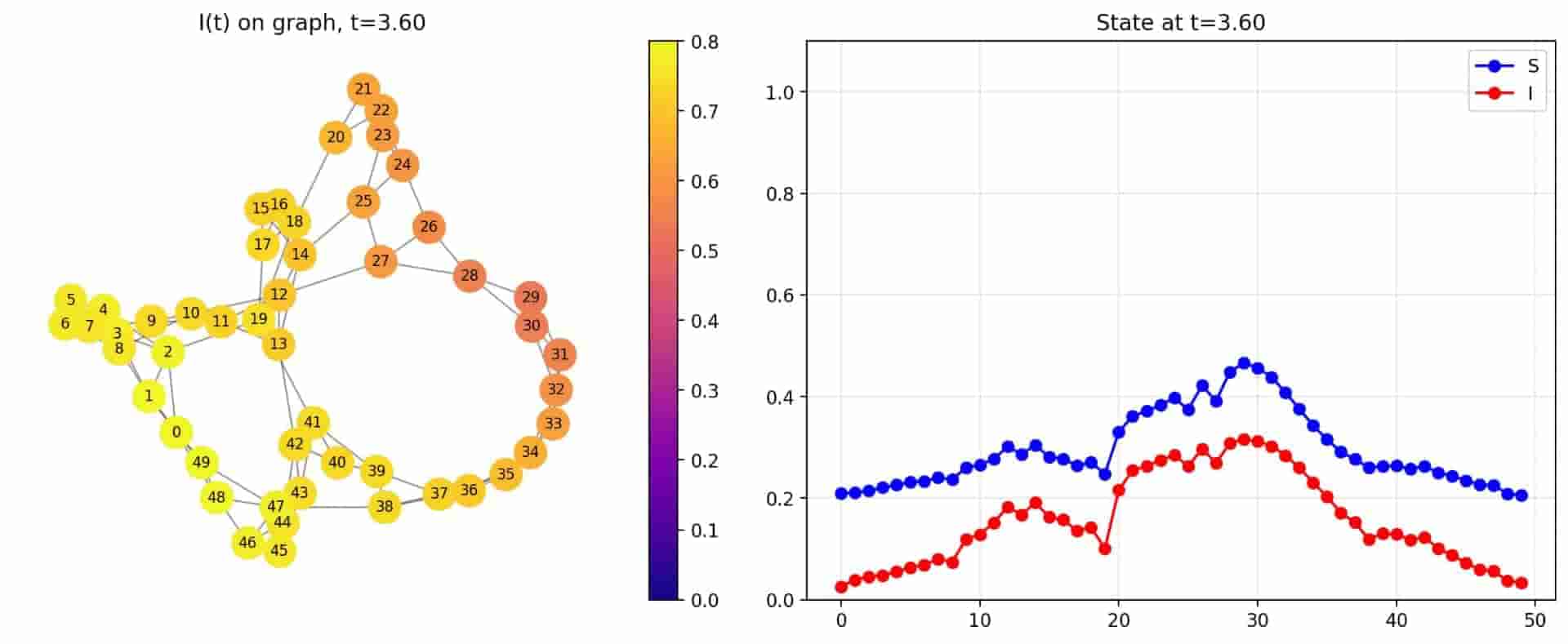}
\end{minipage}

\vspace{0.5cm}

% -------- Row 3 (centered last image) --------
\begin{minipage}{0.48\textwidth}
\centering
\includegraphics[width=\linewidth]{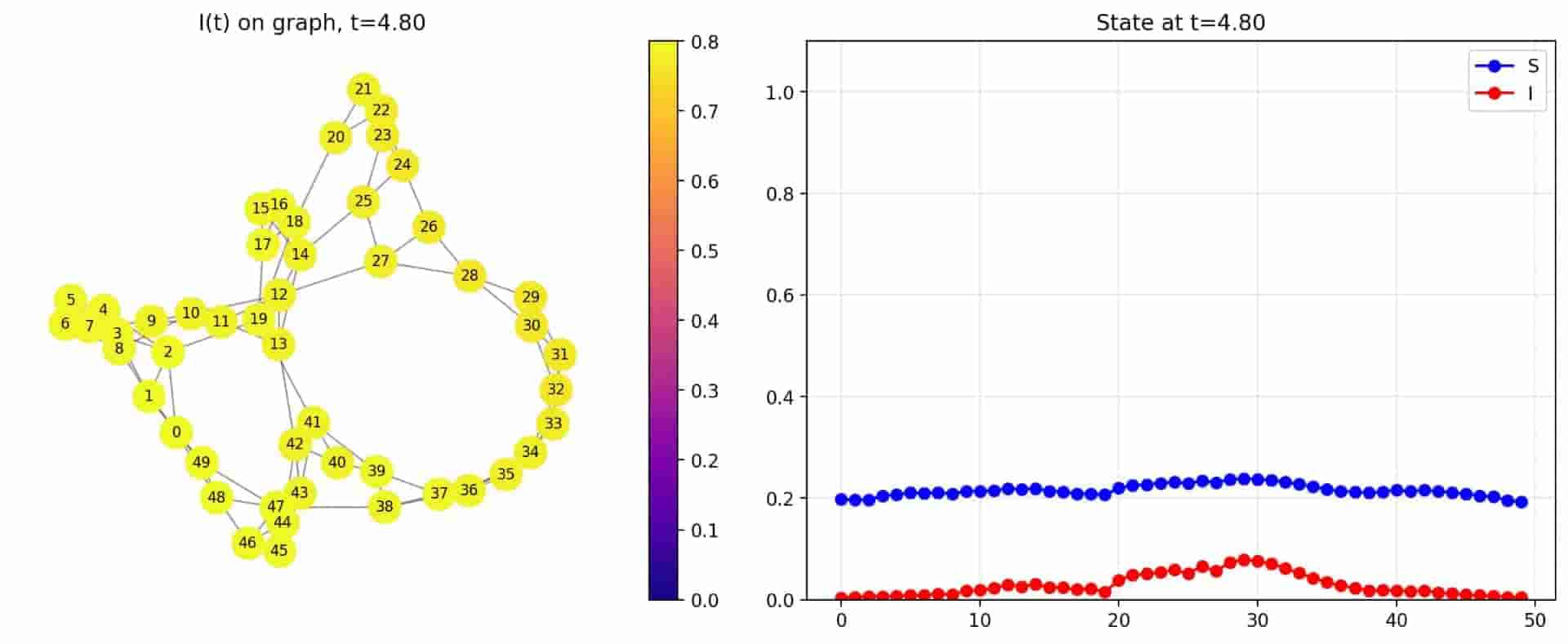}
\end{minipage}
\hfill
\begin{minipage}{0.48\textwidth}
\centering
\includegraphics[width=\linewidth]{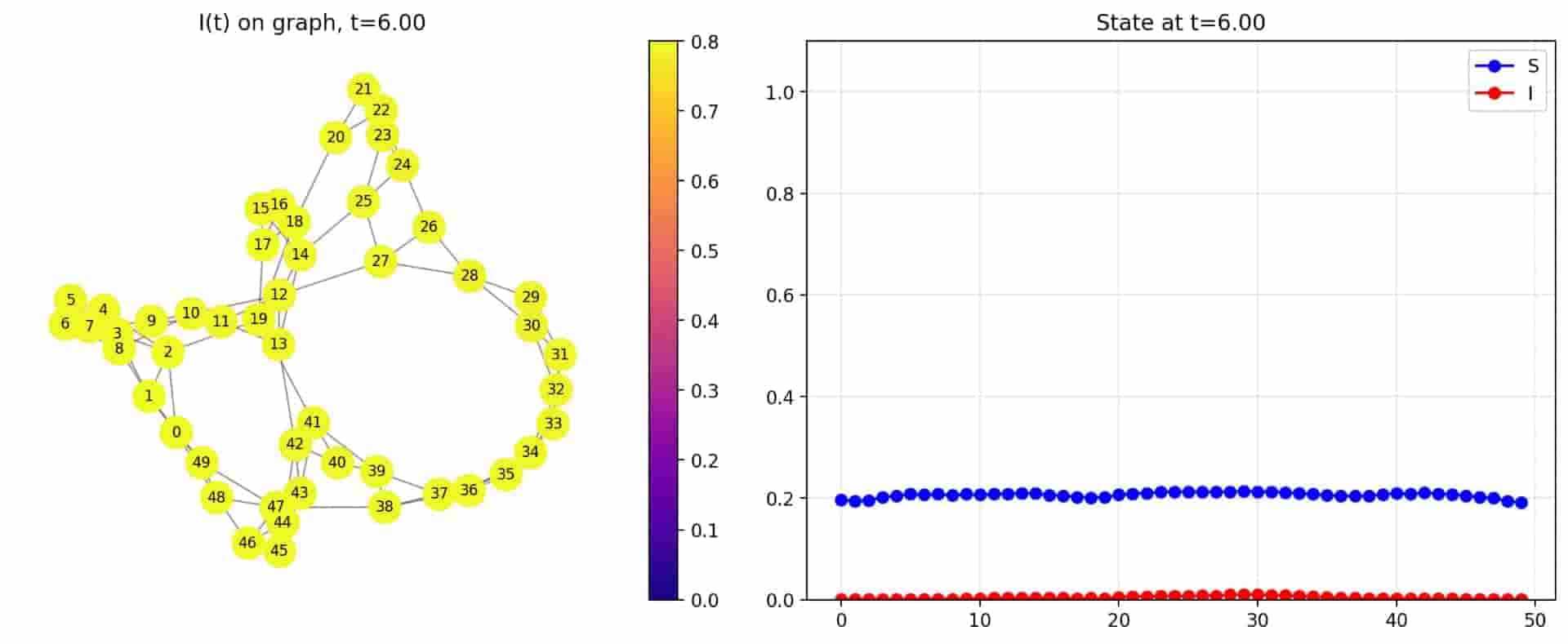}
\end{minipage}

\caption{Propagation of the opinion on a Watts-Strogatz graph. Local stifling.}
\end{figure}

We see that the opinion propagates (starting from the node $n=0$, it eventually reaches all the nodes), and it spreads in a wave-like motion.

Observe also that the final rate of adoption $r_n^\infty$ is independent of $n$ (at least for $n$ far away from $0$), this is in no way generic, see Remark \ref{conj hom} below for a discussion on this topic. 

Here is another simulation, where $G$ is a cyclic graph. The wave-like propagation pattern is clearly visible in this case.

\begin{figure}[H]
\centering

% -------- Row 1 --------
\begin{minipage}{0.48\textwidth}
\centering
\includegraphics[width=\linewidth]{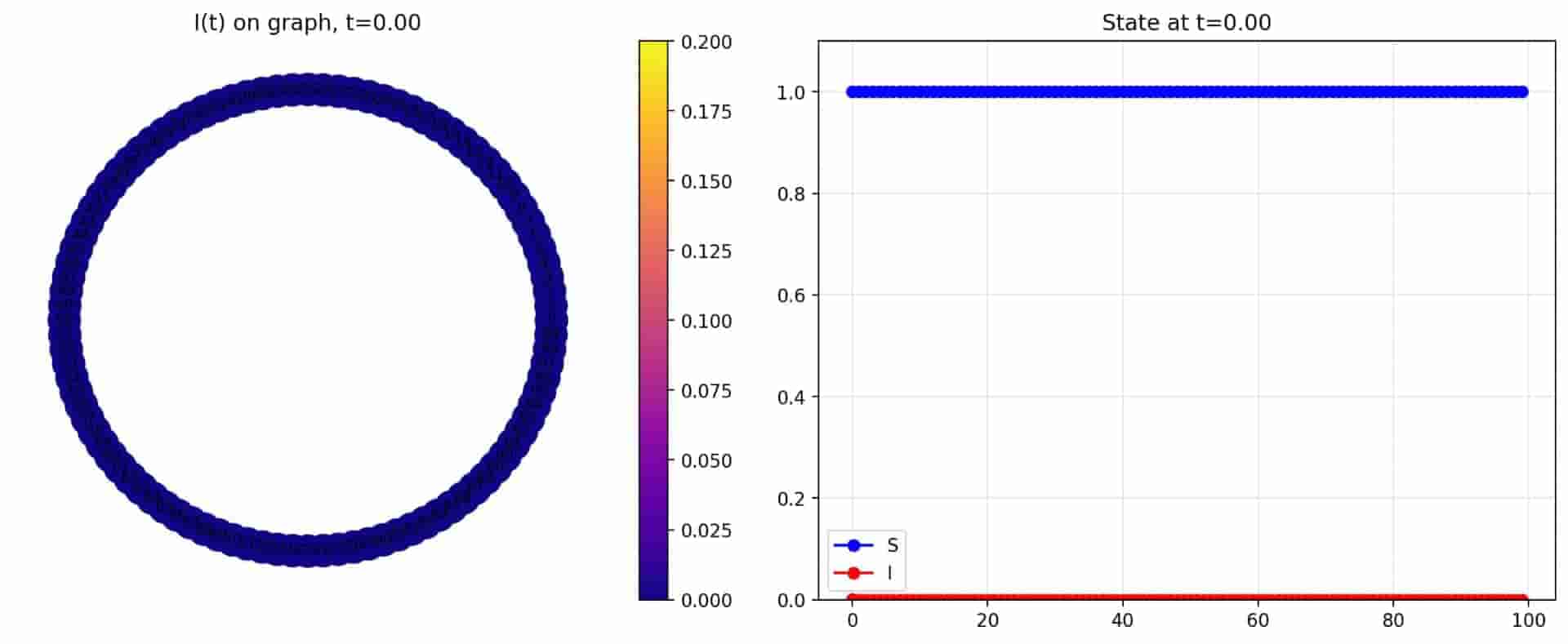}
\end{minipage}
\hfill
\begin{minipage}{0.48\textwidth}
\centering
\includegraphics[width=\linewidth]{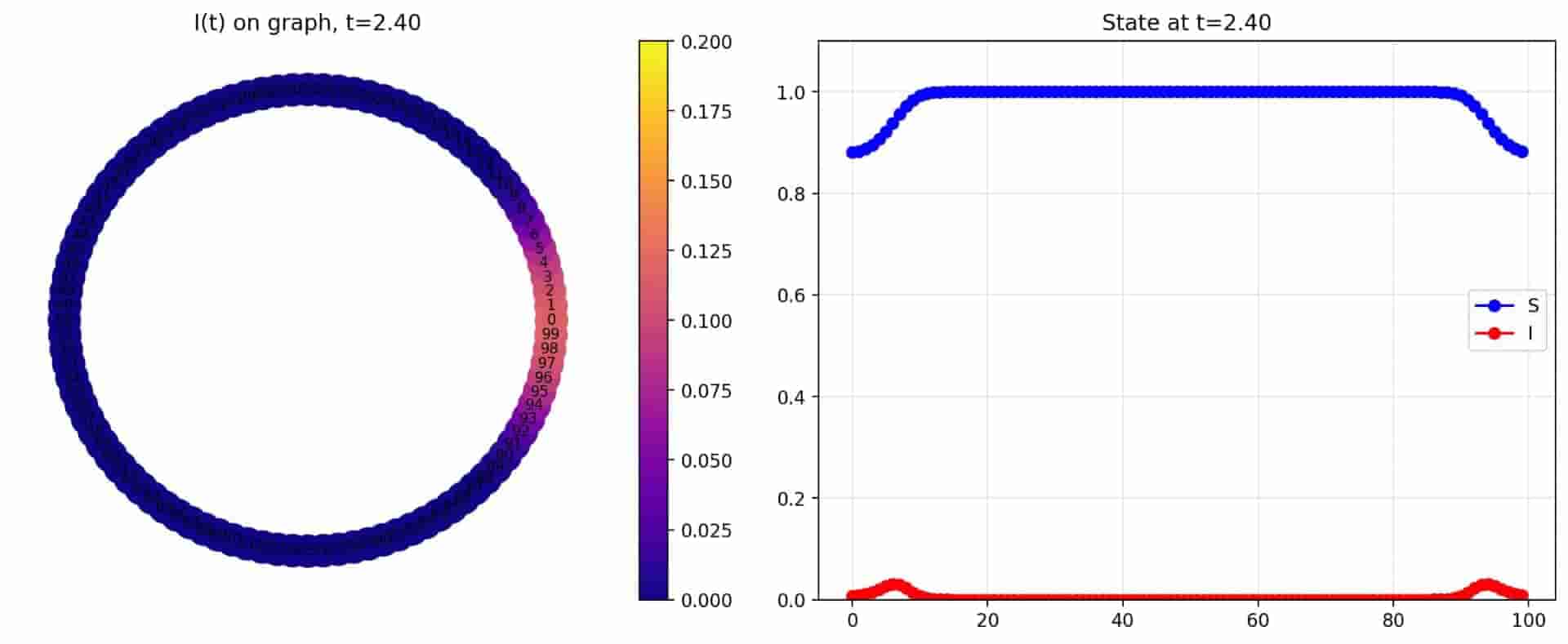}   
\end{minipage}

\vspace{0.5cm}

% -------- Row 2 --------
\begin{minipage}{0.48\textwidth}
\centering
\includegraphics[width=\linewidth]{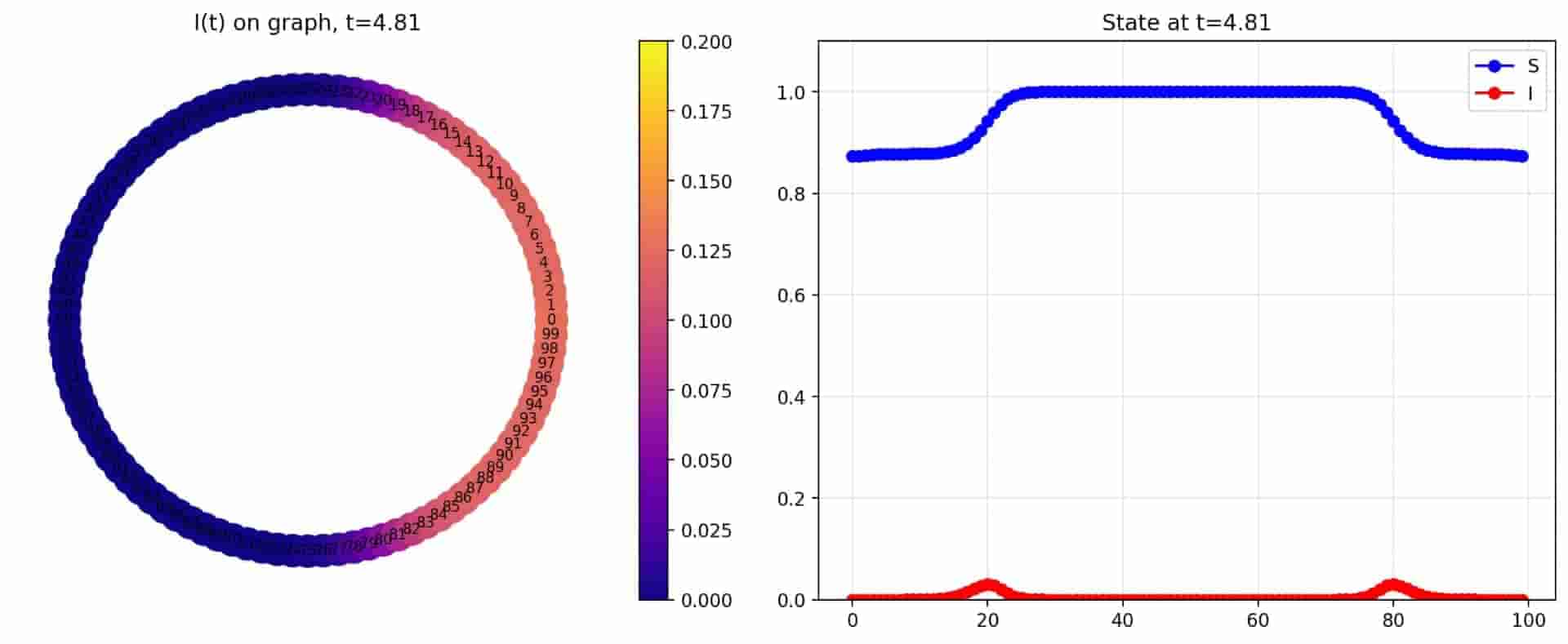}
\end{minipage}
\hfill
\begin{minipage}{0.48\textwidth}
\centering
\includegraphics[width=\linewidth]{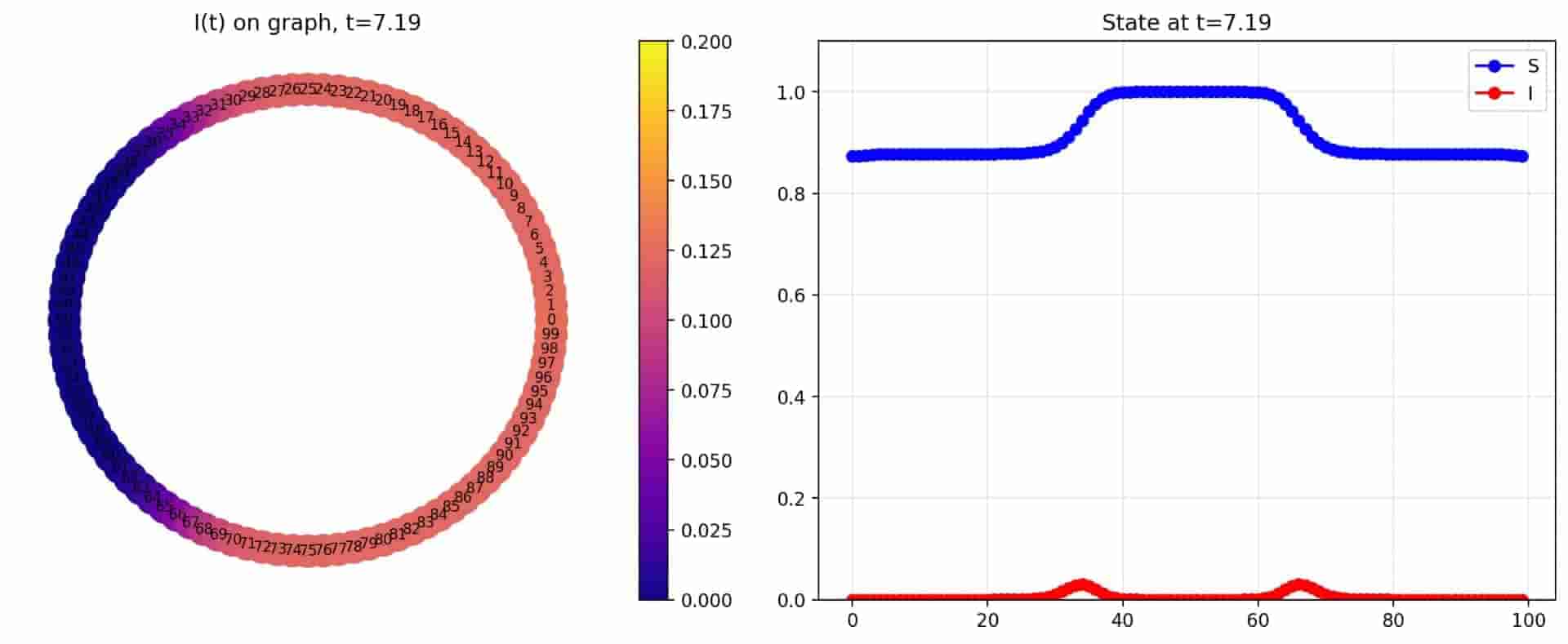}
\end{minipage}

\vspace{0.5cm}

% -------- Row 3 (centered last image) --------
\begin{minipage}{0.48\textwidth}
\centering
\includegraphics[width=\linewidth]{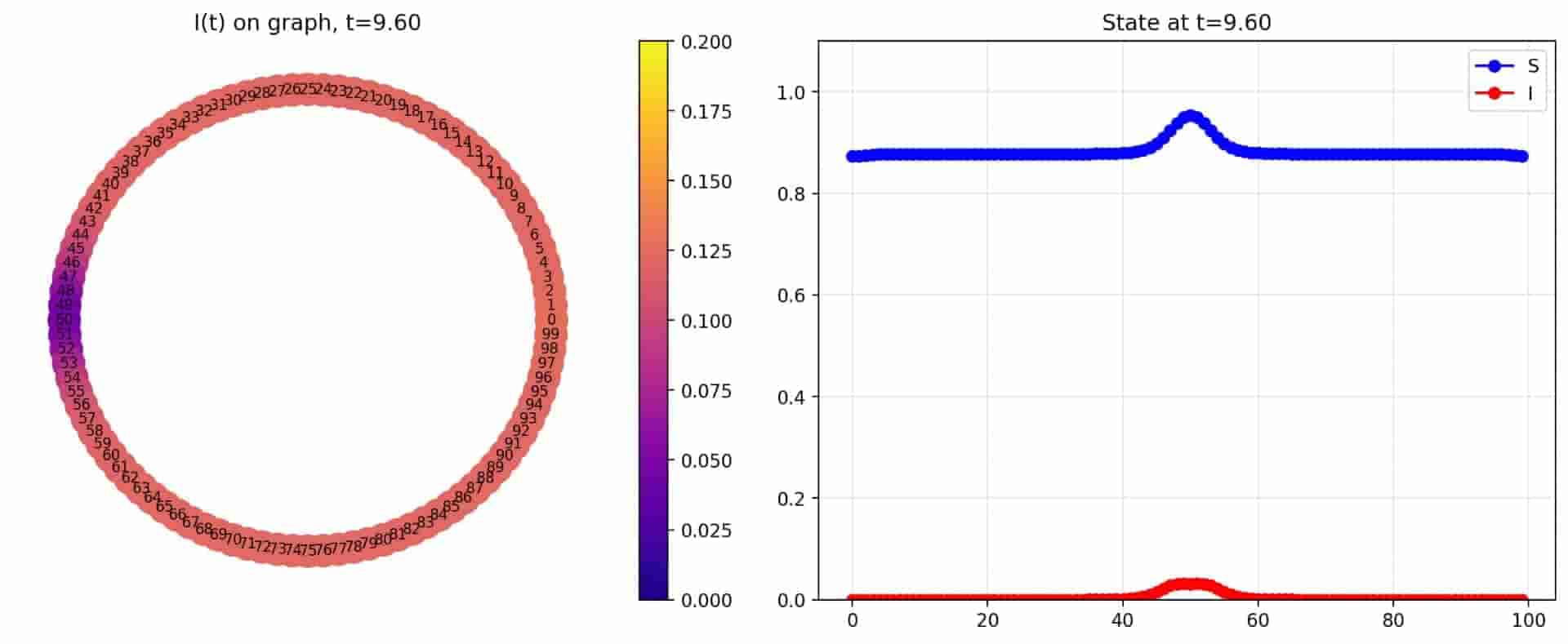}
\end{minipage}
\hfill
\begin{minipage}{0.48\textwidth}
\centering
\includegraphics[width=\linewidth]{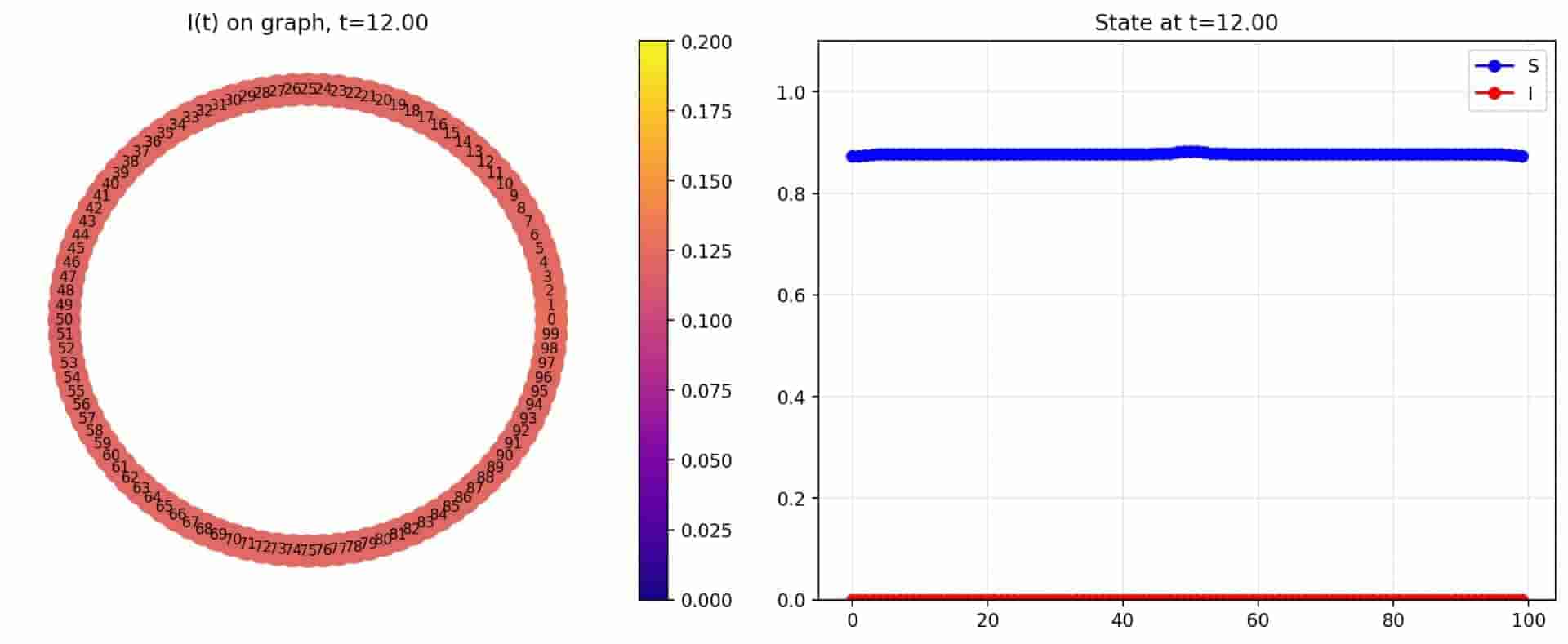}
\end{minipage}

\caption{Propagation of the opinion on a cycle - local stifling}
\end{figure}

Both these simulations are done under the hypothesis of local stifling: we shall primarily study this situation. \\

In the second part of the paper, we consider the case where the stifling can be non-local. We will see that our result concerning the convergence partially persists in this case.

However, we will show some numerical simulations that strongly suggest that the non-local stifling can trigger the emergence of {\em patterns.}

More precisely, we will observe that there are situations where the final distribution of the opinion on the graph is homogeneous when the range of the stifling is ``small'' and becomes heterogeneous when the range of stifling increases. It will appear that long-range stifling may lead to the emergence of segregated ``bubbles'' where the opinion is adopted.

From the sociological point of view, this could be interpreted by saying that strong self-censorship would lead to the generation of isolated ideological ``echo chamber''. 
This can be interesting from the modeling point of view: it suggests a mechanism that could lead to the emergence of ``bubbles'' in the opinion world. This is illustrated in the simulation below, where we use the same parameters as above except with a non-local stifling.

\begin{figure}[H]
\centering

% -------- Row 1 --------
\begin{minipage}{0.48\textwidth}
\centering
\includegraphics[width=\linewidth]{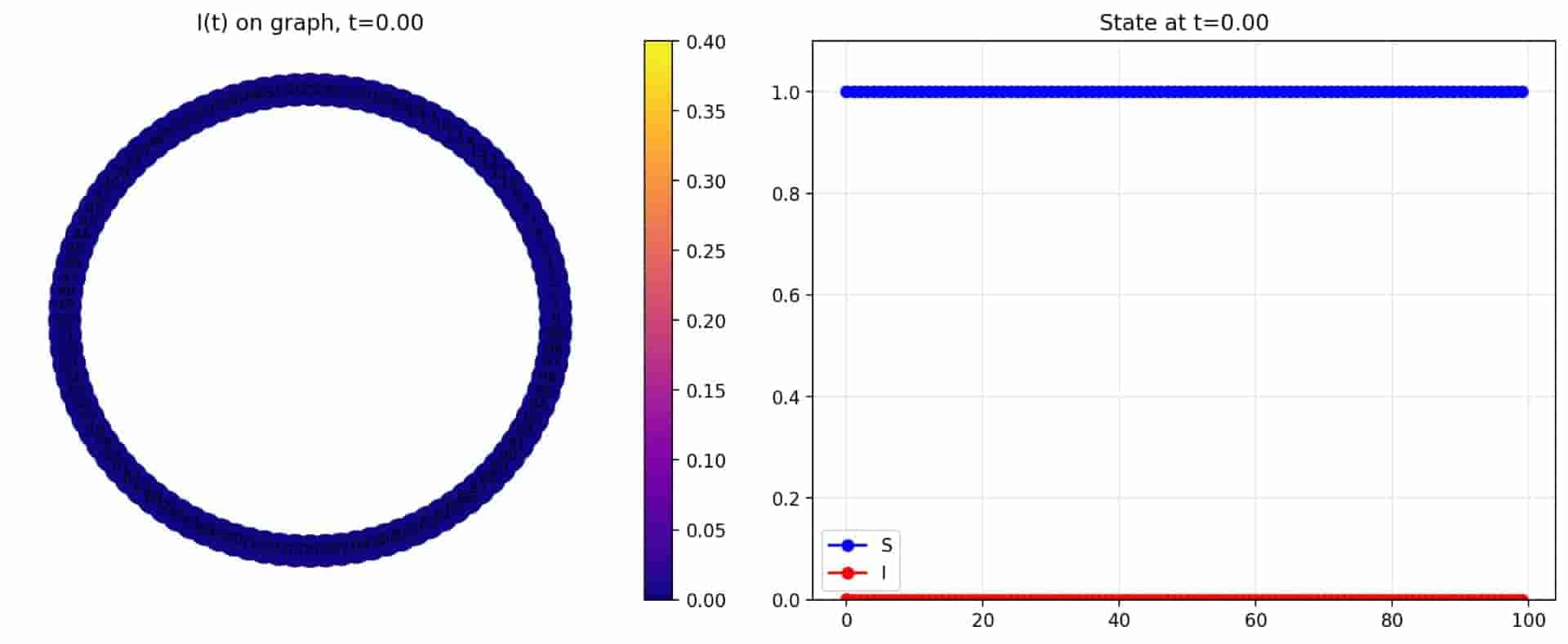}
\end{minipage}
\hfill
\begin{minipage}{0.48\textwidth}
\centering
\includegraphics[width=\linewidth]{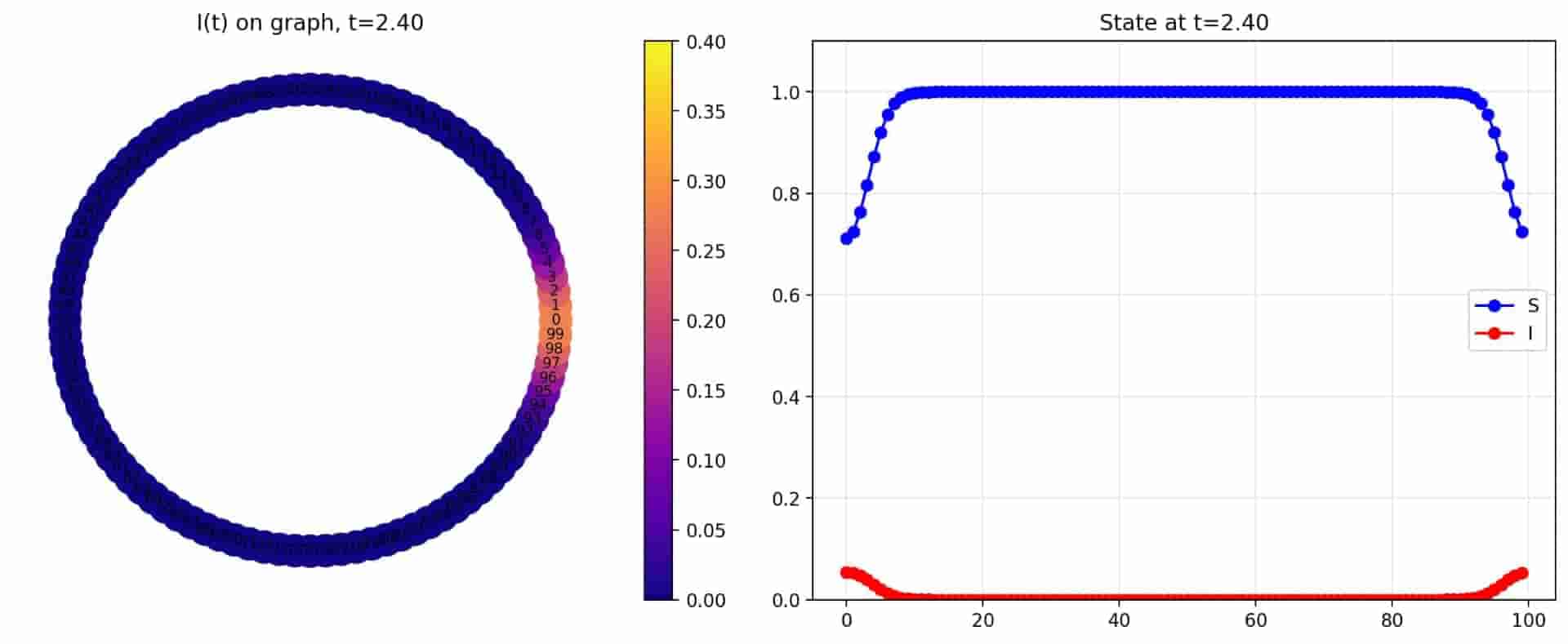}   
\end{minipage}

\vspace{0.5cm}

% -------- Row 2 --------
\begin{minipage}{0.48\textwidth}
\centering
\includegraphics[width=\linewidth]{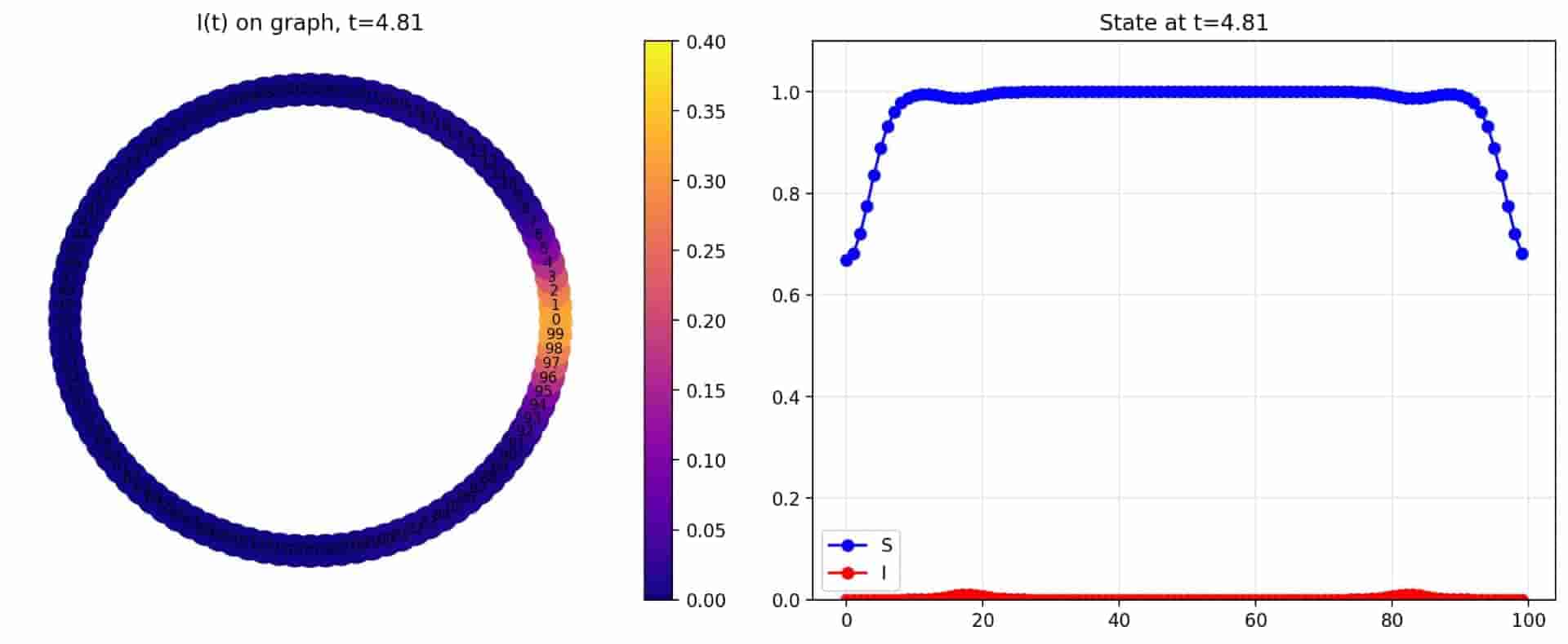}
\end{minipage}
\hfill
\begin{minipage}{0.48\textwidth}
\centering
\includegraphics[width=\linewidth]{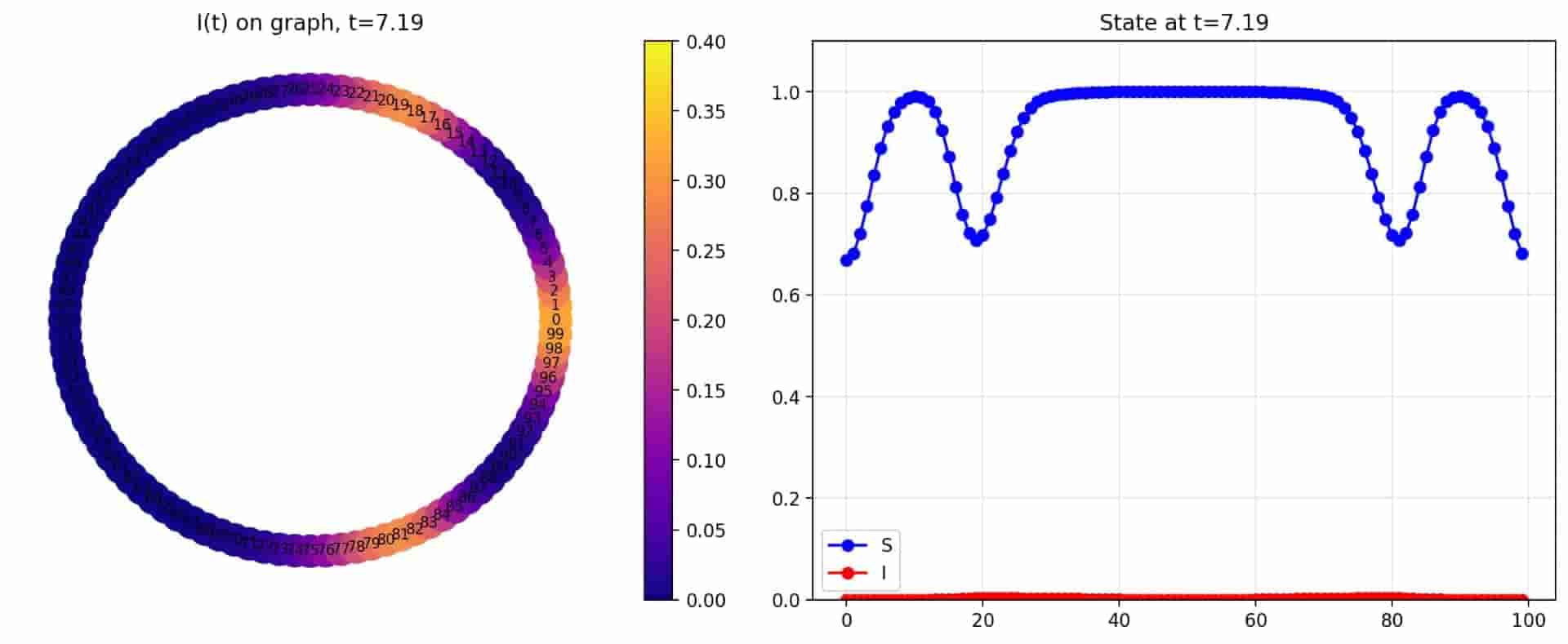}
\end{minipage}

\vspace{0.5cm}

% -------- Row 3 (centered last image) --------
\begin{minipage}{0.48\textwidth}
\centering
\includegraphics[width=\linewidth]{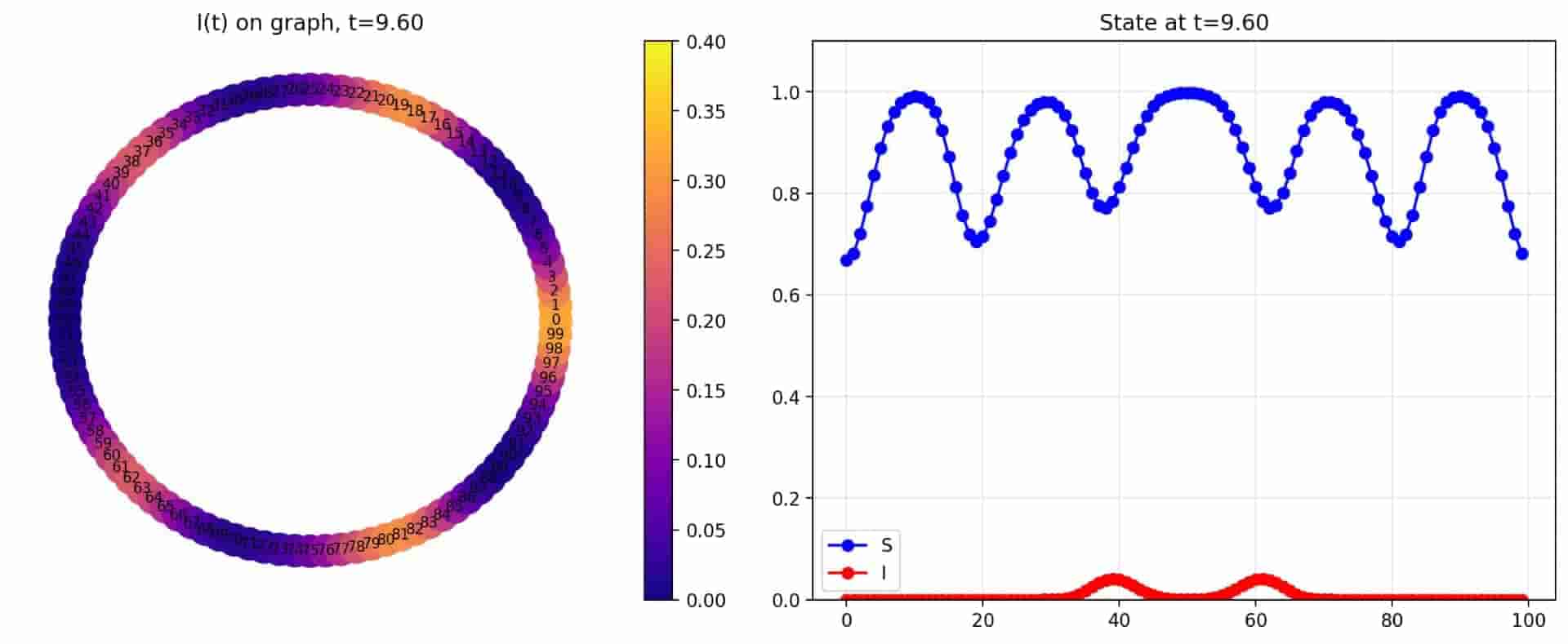}
\end{minipage}
\hfill
\begin{minipage}{0.48\textwidth}
\centering
\includegraphics[width=\linewidth]{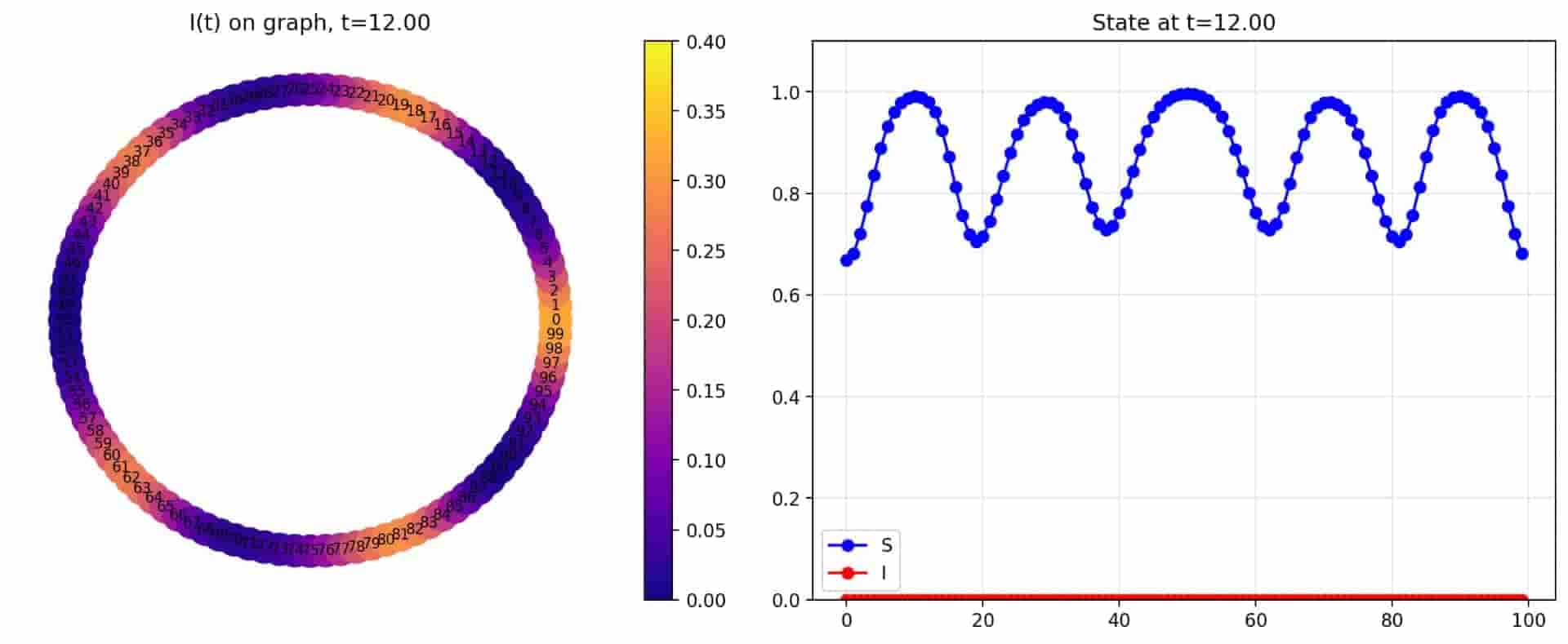}
\end{minipage}

\caption{Propagation of the opinion on a cycle. Non-local stifling.}
\end{figure}

\subsection{Related works}

Let us briefly gather in this section some remarks concerning related works.

First, we want to mention that the system with local stifling \eqref{syst loc} was considered in the paper \cite{NEKOVEE2007457}, where the authors obtain the system \eqref{syst loc} from an individual-based model using the framework of Interactive Markov Chain, and in the paper \cite{moreno2004dynamics}, where the system is studied numerically.

The Daley-Kendall model was also considered in \cite{rumorpei2026numerical} but in a different setting: the authors consider there a continuous version, where the ignorants, spreaders and stiflers move (diffuse) and where the transmission of the opinion is local. The authors conduct numerical experiments. 

    As far as we are aware, the system with non-local stifling \eqref{syst} was not considered.\\

The analysis conducted in the present paper is inspired by the theory of reaction-diffusion equations (in particular applied to epidemiological systems). We refer to \cite{berestycki2020modeling, berestyckiCriminal, ducasse2025emergence} for other examples of reaction-diffusion theory applied to sociological models, but in a continuous setting. We also mention the work \cite{besse2021spreading} where an epidemiological model on a graph is considered.\\

    Finally, we mention that there exist many other deterministic models that describe the spread of opinions, such as the {\em bounded confidence models}, see \cite{ben2015pattern, hegselmann2005opinion} and the references therein for instance.

\subsection{Hypotheses}\label{sec hyp}

We gather in this section the hypotheses that shall be assumed throughout the whole paper.\\

{\em The graph $G$}.
\begin{itemize}
    \item In the whole paper, $G$ is a countably infinite graph, whose nodes are indexed by $n\in\N$.

    \item We assume that the graph $G$ is connected.

    \item The number of neighbors that one node can have is bounded, that is, there is $V \in \N$ such that, for all $i\in \N$, we have
    $$
    \#\{j \ : \ j\sim i\}\leq V.
    $$
    \item In the whole paper, we denote $d(i,n)$ the graph distance between nodes $i$ and $n$. For notational convenience, we also denote $d(n) := d(0,n)$ the distance between the node $0$ and the node $n$. When considering the spreading of the solutions it is convenient to fix a reference point ($0$ here).
\end{itemize}

{\em Hypotheses on the transmission coefficients $(\alpha_{k,n})$.}

\begin{itemize}

\item We have $\alpha_{k,n} \geq 0$ for all $k,n\in\N$.

\item We have $\alpha_{k,n}\neq 0$ if and only if $k\sim n$.

    \item There are $\ul \alpha,\ol \alpha >0$ such that, for all $(k,n)\in \N^2$ such that $k\sim n$, we have $\ul \alpha \leq \alpha_{k,n} \leq \ol  \alpha$.

    \item For all $n\in \N$, we have $\alpha_{n,n} \geq \ul \alpha$ (this is actually a consequence of the previous hypotheses and our convention that $n\sim n$ for all $n\in \N$).
\end{itemize}

{\em Hypotheses on the stifling parameters.} 
\begin{itemize}
    \item For all $(k,n)\in \N^2$, we have $\beta_{k,n}\geq 0$.
    \item There are $\ul \beta,\ol\beta >0$ such that, for all $(k,n)\in\N^2$ such that $\beta_{k,n}\neq 0$, we have $\ul \beta \leq \beta_{k,n}\leq \ol \beta$.
    \item For all $n\in \N$, we have $\beta_{n,n} >0$.

    \item There is $R_\beta\geq 0$ such that $\beta_{k,n}\neq 0$ only if $d(k,n)\leq R_\beta$. 
\end{itemize}

The parameter $R_\beta$ is the range of the stifling process. When $R_\beta =0$, the stifling is local.\\

{\em Hypotheses on the initial data.}

The evolution system \eqref{syst} (or \eqref{syst loc}) will be completed with initial data $(S_n^0,I_n^0,R_n^0)$, for all $n\in\N$. As already mentioned, we always take $R_n^0 = 0$, for all $n\in \N$.  We also assume that
\begin{itemize}
    \item There are $\ul \sigma, \ol \sigma, \ol \iota >0$ such that 
    $$
\ul \sigma \leq S_n^0 \leq \ol \sigma,\quad \forall n\in \N,
$$
and
$$
0\leq I_n^0 \leq \ol \iota,\quad \forall n\in \N.
$$
\item We shall also assume that the initial population of spreaders is compactly supported, that is, there are only a finite number of communities where the opinion is present at initial time:
$$
\#\{n \ : \ I_n^0 >0\}<+\infty.
$$
\end{itemize}

Under all these hypotheses, it is standard to show that there exist classical solutions to the system \eqref{syst}, and that these solutions are unique and non-negative, i.e., $S_n,I_n,R_n\geq 0$ for all $n\in\N$.

\subsection{Results}

We state in this section our main results concerning the model \eqref{syst}. Our results concern the long-time behavior of the solutions $(S_n,I_n)$ of \eqref{syst}, they are stated in the theorems below. After each theorem, we give some interpretation using the rate of adoption \eqref{rate}.\\  

We first consider the model \eqref{syst loc} with local stifling. In this case, we generalize the result of Daley and Kendall \cite{DaleyKendall} that says that the opinion always spreads (there is no threshold effect such as the one observed in epidemiology), but only a portion of the population eventually adopts the opinion. We give upper and lower estimates on this proportion.

\begin{theorem}\label{th cv}
    Consider the case with purely local stifling, system \eqref{syst loc}. Let $(S_n,I_n)$ be the solution of \eqref{syst loc} arising from an initial datum $(S_n^0,I_n^0)$ satisfying the hypotheses of Section \ref{sec hyp}.\\

    Then, for all $n\in \N$,
    $$
    I_n(t)\underset{t\to+\infty}{\longrightarrow}0, \quad S_n(t)\underset{t\to+\infty}{\longrightarrow}S_n^\infty,
    $$
    where we have, for all $n\in \N$,
   $$
   S_n^0 e^{-\sum_i \frac{\alpha_{i,n}}{\beta_i}}e^{-\sum_i \frac{\alpha_{i,n}}{\alpha_{i,i}}}\leq S_n^\infty \leq S_n^0 e^{-\sum_i \frac{\alpha_{i,n}}{\beta_i}}.
    $$
\end{theorem}

\noindent {\bf Interpretation.} Theorem \ref{th cv} tells us that the opinion always spreads through the whole network. Indeed, defining $\ul r = 1- \exp(-V\frac{\ul  \alpha}{\overline\beta})$ and $\ol r = 1- \exp(-V\frac{\overline \alpha}{ \underline \beta + \ul \alpha})$ (where $\ol \alpha,\ul  \alpha, V$ are from the hypotheses in Section \ref{sec hyp}), the theorem tells us that, for all $n\in \N$ such that $I_n^0=0$ (that is, all nodes except a finite number), the rate of adoption $r_n(t)$ converges when $t$ goes to $+\infty$ to some value $r_n^\infty$ such that
$$
0<\ul r \leq r_n^\infty\leq \ol r <1.
$$
This means that, even though the nodes where the opinion is present at initial time (which are the nodes $n$ such that $I_n^0>0$) are in finite number, all the nodes will eventually adopt the opinion, but only partially (the rate of adoption is uniformly strictly larger than $0$ and smaller than $1$). \\

Without supplementary hypotheses on the coefficients or on the topology of the graph, we can not expect to get much more information on $r_n^\infty$. However, we believe that the situation can be made explicit when the model is {\em homogeneous}, see Remark \ref{conj hom} below.\\

Now that we know that the opinion always propagates, we study its {\em speed of propagation}. Indeed, we have shown that even the nodes which are very far away will eventually adopt (in some proportion) the opinion. However, it is reasonable to think that, the further the nodes, the longer it will take for the opinion to reach it. This can be quantified.

We still consider the case with local stifling, \eqref{syst loc} and we give some estimates on the time at which each node will adopt the opinion.

\begin{theorem}\label{th speed}
Let $(S_n,I_n)$ be the solution of \eqref{syst loc} with initial datum $(S_n^0,I_n^0)$ that satisfies the hypotheses of Section \ref{sec hyp}.
Let 
$$
c^\star = \inf_{\lambda>0}\sup_n \frac{S_n^0\sum_i \alpha_{i,n}e^{\lambda (d(n)-d(i))}}{\lambda}
$$
and
$$
c_\star = \inf_{\substack{ n,m\in\N^2 \\ d(n,m)\leq1}} \frac{\beta_m}{\beta_n}\alpha_{n,m}S_m^0.
$$
Owing to the hypotheses in Section \ref{sec hyp}, these quantities are finite and strictly positive.\\

    Then, the opinion spreads at most with speed $c^\star$ and at least with speed $c_\star$ in the following sense:
$$
\sup_{d(n)\geq c t}\vert S_n(t) - S_n^0\vert \underset{t\to +\infty}{\longrightarrow}0,\quad \forall c>c^\star,
$$
and
$$
\limsup_{t\to+\infty}\sup_{d(n)\leq ct} (S_n(t) - S_n^0 e^{-\sum_i \frac{\alpha_{i,n}}{\beta_i}}) \leq 0,\quad \forall c \in (0,c_\star).
$$
\end{theorem}

\noindent {\bf Interpretation.} Theorem \ref{th speed} tells us how fast the rate of adoption evolves through the graph. Indeed, it implies that
$$
\sup_{d(n)\geq ct} r_n(t) \underset{t\to+\infty}{\longrightarrow}0,\quad \forall c >c^\star
$$
and
$$
\liminf_{t\to + \infty}\inf_{d(n)\leq c t} r_n(t) \geq \ul r,\quad \forall c \in [0,c_\star[,
$$
where $\ul r$ is the lower bound for $r_n^\infty$ from the interpretation of Theorem \ref{th cv}.\\

In other words, the opinion spreads at speed at most $c^\star$ and at least with speed $c_\star$: indeed, imagine that there is an observer moving from one node to another on the graph. Assume that this observer starts at the node $n=0$ and always moves far away from this point with constant speed $c>c^\star$. By time $t$, he would have reached some node $n_t$, with $d(n_t,0) \approx ct \gg c^\star t$. Owing to Theorem \ref{th speed}, the rate of adoption at this node $r_n(t)$ would be close to $0$, that is, the opinion has not yet reached this point, and the observer will outrun the spread of the opinion.

On the other hand, if the observer moves with a speed $c<c_\star$, then the rate of adoption of the opinion will be at least $\ul r>0$ in his surroundings.\\

Again, without further hypotheses on the topology of the graph and on the parameters, one can not expect to get much more precise results on the spreading. However, we believe that, when the system is {\em homogeneous}, one could obtain more precise results.

\begin{remark}[Homogeneous case]\label{conj hom}
We say that the model is {\em homogeneous} when the following conditions hold true: there are $S^0,a,b>0$ and $R_\beta\geq 0$ such that
\begin{itemize}
    \item $S_n^0 = S^0$ for all $n\in \N$,
    \item $\alpha_{i,n} =\frac{a}{v_n}$, for all  $i,n$ such that $i\sim n$, where $v_n=\#\{i \ : \  i\sim n\}$ is the number of nodes adjacent to $n$ (recall that by convention in this paper, the node $n$ is always adjacent to itself).
    \item $\beta_{i,n} = \frac{b}{\rho_n}$, for all $i,n$  such that $d(i,n)\leq R_\beta$, where $\rho_n = \#\{i \ : \ d(i,n)\leq R_\beta\}$.
\end{itemize}

In this setting, the initial population is homogeneous, and each community is influenced similarly by each neighboring community (these assumptions were made in the simulation shown above in Section \ref{sec  DK net}). In this case, the system \eqref{syst} rewrites as

\begin{equation*}
\left\{
\begin{array}{rll}
     S'_n(t) &= - S_n(t)\left(\frac{a}{v_n}\sum_{i\sim n} I_i(t) \right),\quad &t>0,\ n\in \N,\\
    I'_n(t) &=  S_n(t)\left(\frac{a}{v_n}\sum_{i\sim n} I_i(t) \right) - I_n(t)\left(\frac{b}{\rho_n}\sum_{i} (P_i - S_i(t)) \right),\quad &t>0,\ n\in \N.
\end{array}
\right.
\end{equation*}
If these homogeneity hypotheses hold true and if the stifling is local (that is, $R_\beta=0$), then we conjecture that there is $S^\infty>0$ (independent of $n$) such that 
    $$
    S_n^\infty \underset{d(n)\to+\infty}{\longrightarrow}S^\infty.
    $$
    We leave the proof of this fact as an open question for future works.
\end{remark}

The formulas for $c_\star,c^\star$ in Theorem \ref{th speed} are rather general. When working with simpler graphs $G$, they become more explicit.

\begin{corollary}
Under the hypotheses of Section \ref{sec hyp} and of Theorem \ref{th speed}, we have the following estimates on the speeds $c_\star,c^\star$ from Theorem \ref{th speed}.
\begin{itemize}
    \item For any graph $G$, we always have the general estimates
    $$
    \frac{\ul \beta}{\ol \beta}\ul \alpha \ul \sigma \leq c_\star \leq c^\star \leq \ol \alpha \ol \sigma V.
    $$

    \item For any graph $G$, if the homogeneity hypotheses of Remark \ref{conj hom} hold true, then we have
    $$
   \frac{1}{V}a  S^0 \leq c_\star  \ \text{and}\   c^\star = \theta a S^0 ,
    $$
    where $\theta = \inf_{\lambda>0}\sup_n \left(\frac{\frac{1}{v_n}\sum_{i\sim n}e^{\lambda(d(n)-d(i))}}{\lambda}\right)$.
    
    \item Consider the case where the graph $G$ is a line and that the homogeneity hypotheses of Remark \ref{conj hom} hold true. 
Then
$$
c_\star =\frac{1}{3}a S^0  \ \text{and }\ c^\star = z a S^0,
$$
where $z = \inf_{\lambda>0} \frac{1+e^\lambda + e^{-\lambda}}{3\lambda}\approx 1.3$.

\end{itemize}

\end{corollary}

\begin{remark}
The formula for $c^\star$ is reminiscent of the formula for the speed of spreading for discrete {\em KPP equations}. These are equations of the form
$$
u'_n(t) = u_{n-1}(t) - 2 u_n(t)+u_{n+1}(t) + u_n(t)(1-u_n(t)),\quad t>0,\ n\in \Z.
$$
Such equations were introduced and studied in the continuous setting by Kolmogorov, Petrovski and Piskunov \cite{KolmogorovStudy37} and Fisher \cite{fisher1937wave}. The discrete setting was considered in \cite{zinner1993traveling} (see also \cite{besse2023logarithmic} for a recent account on this topic). In this case, one can prove that the solution converges to $1$ and propagates with a speed
$$
c_{KPP} = \min_{\lambda>0}\frac{e^\lambda - 1 + e^{-\lambda}}{\lambda}.
$$
\end{remark}

We conclude this paper with a study of the model \eqref{syst} with non-local stifling. We show that, even in this case, the opinion always spreads in a weaker sense: any zone where the opinion is never adopted can not be ``too large''.
\begin{theorem}\label{th spread non loc}
Let $(S_n,I_n)$ be the solution of \eqref{syst} with initial datum $(S_n^0,I_n^0)$ satisfying the hypotheses of Section \ref{sec hyp}.

    Then, for all $n\in \N$,
    $$
    I_n(t)\underset{t\to+\infty}{\longrightarrow}0, \quad S_n(t)\underset{t\to+\infty}{\longrightarrow}S_n^\infty,
    $$
    and
    $$
    \liminf_{d(n)\to+\infty} \frac{S_n^\infty}{S_n^0}<1.
    $$
\end{theorem}

\noindent {\bf Interpretation.} This result shows that the opinion still propagates, but in a sense weaker than in Theorem \ref{th cv}. Indeed, Theorem \ref{th spread non loc} tells us that $r_n^\infty$, the final rate of adoption, is such that $\limsup_{d(n)\to +\infty} r_n^\infty >0$. This means that there will be some nodes $n$, with $d(n)$ arbitrarily large (that is, nodes arbitrarily far away from the initial nodes where the opinion is present at time $t=0$) where the opinion is eventually known. 

It is weaker than Theorem \ref{th cv} in the sense that there may be some nodes $n$ where $r_n^\infty \approx 0$, that is, nodes where the opinion is never adopted.\\

We will conclude this paper with an open question, that suggests that the non-local stifling can induce the emergence of patterns in the solutions.

\begin{conj}\label{conj patterns}
    Consider the system \eqref{syst} under the homogeneity hypotheses of Remark \ref{conj hom}. In particular, we have $\beta_{i,n} = \frac{b}{\rho_n}$, $\forall i,n$  such that $d(i,n)\leq R_\beta$, where $\rho_n = \#\{i \ : \ d(i,n)\leq R_\beta\}$.

    We conjecture that there are some values of $a,b, R_\beta$ such that 
    $$
    \limsup_{n\to+\infty}S_n^\infty> \liminf_{n\to+\infty} S_n^\infty.
    $$
    
\end{conj}

The previous Remark \ref{conj hom} conjectures that, if $R_\beta=0$ (the stifling is purely local), then $S_n^\infty\underset{d(n)\to+\infty}{\longrightarrow} S^\infty$ for some $S^\infty$: the opinion is adopted in the same proportion in each node (at least far away from the zone where the opinion is initially present). Conjecture \ref{conj patterns} claims that this becomes false when the individuals take into account the nodes that are far away in their stifling mechanism. 

We present in Section \ref{sec non loc} some simulations that illustrate this conjecture, that we leave as an open problem.

\section{The case with purely local stifling}

The goal of this section is to prove Theorem \ref{th cv} and Theorem \ref{th speed}.

\subsection{Rewriting the system}

A key idea in our analysis will be to work with auxiliary functions rather than with the functions $(S_n,I_n)$ solutions of \eqref{syst loc}. We define
\begin{equation}\label{def u}
   u_n(t)=\int_0^t I_n(\tau)d\tau,\quad n\in \N.
\end{equation}
These new functions represent in some sense the ``history'' of the propagation of the opinion. In mathematical epidemiology, these functions are called  the {\em force of infection}, and it is classical to work with them.\\

We start with proving that the functions $(u_n)_{n\in\N}$ are solutions of a system of ODEs with non-local in time terms. For later use, we state the result for the general model \eqref{syst} with non-local stifling. We specify the result to the local stifling case \eqref{syst loc} just after.

\begin{lemma}\label{lemma un}
Let $(S_n,I_n)$ be the solution of \eqref{syst} arising from an initial datum $(S_n^0,I_n^0)$ that satisfies the hypotheses of Section \ref{sec hyp}. The functions $(u_n)_{n\in\N}$ defined by \eqref{def u} satisfy, for all $t>0$,
    \begin{multline}\label{eq u}
u_n'(t) = S_n^0f\left(\sum_k \alpha_{k,n}u_k(t)\right) - \int_0^t u_n'(\tau)\left(\sum_i \beta_{i,n}S_i^0 f\left(\sum_k \alpha_{k,i}u_k(\tau)\right) \right)d\tau \\ - u_n(t)\left(\sum_i \beta_{i,n}I_i^0\right) + I_n^0.
\end{multline}
    where 
    $$
    f(x) = 1-e^{-x}.
    $$
\end{lemma}
\begin{proof}
The equation for $I_n$ in \eqref{syst} rewrites (remembering that $P_i = S_i^0+I_i^0$)
$$
u_n''(t) = -S_n'(t) - u_n'(t)\left(\sum_i \beta_{i,n}(S_i^0 - S_i(t)) \right) - u_n'(t)\left(\sum_i \beta_{i,n}I_i^0\right).
$$
On the other hand, the equation for $S_n$ in \eqref{syst} yields that $S_n'(t) = -S_n(t)(\sum_i \alpha_{i,n}u_i'(t) )$, which we can integrate to find
\begin{equation}\label{Su}
    S_n(t) = S_n^0\exp\left(-\sum_i \alpha_{i,n}u_i(t)\right).
\end{equation}
Using this expression in the equation above, we get
$$
u_n''(t) = -S_n'(t) - u_n'(t)\left(\sum_i \beta_{i,n}S_i^0\left( 1- \exp\left(-\sum_k \alpha_{k,i}u_k(t)\right)\right) \right) - u_n'(t)\left(\sum_i \beta_{i,n}I_i^0\right)
$$
Now, integrating between time $t=0$ and $t$, and using $f(x)=1-e^{-x}$, this gives
$$
u_n'(t) = S_n^0-S_n(t) - \int_0^t u_n'(\tau)\left(\sum_i \beta_{i,n}S_i^0f\left(\sum_k \alpha_{k,i}u_k(\tau)\right) \right)d\tau - u_n(t)\left(\sum_i \beta_{i,n}I_i^0\right) + I_n^0,
$$
and, using the expression \eqref{Su} again, we get the result.
\end{proof}
In particular, when we consider the case with local stifling \eqref{syst loc} (that is, when the parameters $(\beta_{k,n})$ satisfy \eqref{beta local}), we have
\begin{corollary} Let $(S_n,I_n)$ be solution of \eqref{syst loc} (the system with local stifling) arising from initial datum $(S_n^0,I_n^0)$ satisfying the hypotheses of Section \ref{sec hyp}. Let $(u_n)$ be defined by \eqref{def u}. We have
    \begin{equation}\label{eq u loc}
    u_n'(t) = S_n^0f\left(\sum_k \alpha_{k,n}u_k(\tau)\right) - \beta_n S_n^0\int_0^t u_n'(\tau) f\left(\sum_k \alpha_{k,n}u_k(\tau) \right)d\tau - u_n(t) \beta_n I_n^0 + I_n^0.    
    \end{equation}
\end{corollary}

The proof of our results will rely on the analysis of equation \eqref{eq u} (or \eqref{eq u loc} when considering the case with local stifling). Before diving into the proof, we mention some basic facts on the functions $(u_n)_{n\in\N}$.

\begin{prop}\label{prop u}
Let $(S_n,I_n)$ be the solution of \eqref{syst} arising from an initial datum $(S_n^0,I_n^0)$ satisfying the hypotheses of Section \ref{sec hyp}. Let $(u_n)$ be defined by \eqref{def u}.

    For each $n\in\N$, the function $t\mapsto u_n(t)$ is non-negative and increasing with respect to $t$.
    
    In addition, it is bounded from above, hence it converges when $t$ goes to $+\infty$.
\end{prop}
\begin{proof}
The fact that each function $u_n(t)$ is increasing with respect to the variable $t$ and is non-negative comes from its definition, $u_n(t)=\int_0^t I_n(\tau)d\tau$, and from the fact the functions $I_n$ are non-negative.\\

Let us prove the boundedness. Observe that, from \eqref{syst or}, we have $R'_n(t) =  I_n(t)\sum_i\beta_{i,n} (I_i(t)+R_i(t)) \geq \beta_{n,n} I_n(t)(I_n(t)+R_n(t))$. Because $u_n'(t)=I_n(t)$, this yields
$$
(R_n(t)e^{-\beta_{n,n}u_n(t)})' \geq \beta_{n,n}(u_n'(t))^2e^{-\beta_{n,n}u_n(t)},
$$
hence, remembering that $R_n\leq P_n$, we have
$$
\beta_{n,n}\int_0^t(u_n'(\tau))^2e^{-\beta_{n,n}u_n(\tau)}d\tau \leq P_n e^{-\beta_{n,n}u_n(t)}.
$$
Therefore, if $u_n$ were not bounded from above, it would go to $+\infty$ as $t$ goes to $+\infty$. Hence, we would have $\int_0^t(u_n'(\tau))^2e^{-\beta_{n,n}u_n(\tau)}d\tau=0$ for all $t>0$, and then, for each $t>0$, $u_n'(t)=0$, which is a contradiction.
\end{proof}

Observe that the upper bound we obtained depends a priori on $n$, and might diverge for large values of $n$. We shall later obtain an explicit, uniform, upper bound for the functions $u_n$.

\subsection{Convergence of the solution in the local stifling case}

In the whole section, we assume that $(S_n,I_n)$ is the solution of the model with local stifling \eqref{syst loc} arising from an initial datum $(S_n^0,I_n^0)$ that satisfies the hypotheses of Section \ref{sec hyp}. We denote $(u_n)$ the functions defined by \eqref{def u}.\\

The goal of this section is to prove Theorem \ref{th cv}. To do so, we prove the following:
\begin{prop}\label{prop cv}
    For each $n$, there is $U_n>0$ such that, we have $u_n(t) \underset{t\to+\infty}{\longrightarrow} U_n$. Moreover, we have
    $$
    \frac{1}{ \beta_n}\leq U_n\leq \frac{1}{\beta_n}+ \frac{1}{\alpha_{n,n}}.
    $$
\end{prop}

Before going to the proof of Proposition \ref{prop cv}, let us explain how it yields Theorem \ref{th cv}.

\begin{proof}[Proof of Theorem \ref{th cv}]
    Let $(S_n,I_n)$ be solution to \eqref{syst loc}. Let $(u_n)$ be defined by \eqref{def u}. It follows from \eqref{Su} that
    $$
    S_n \underset{t\to+\infty}{\longrightarrow} S_n^0 e^{-\sum_i\alpha_{i,n}U_i}.
    $$
    Denoting $S_n^\infty := S_n^0 e^{-\sum_i\alpha_{i,n}U_i}$, we have, thanks to Proposition \ref{prop cv},
    $$
    S_n^0 e^{-\sum_i \frac{\alpha_{i,n}}{\beta_i}}e^{-\sum_i \frac{\alpha_{i,n}}{\alpha_{i,i}}}\leq S_n^\infty \leq S_n^0 e^{-\sum_i \frac{\alpha_{i,n}}{\beta_i}}.
    $$
    This proves the convergence and the estimates on $S_n(t)$.\\

    Let us show that $I_n(t)\underset{t\to+\infty}{\longrightarrow}0$. To do so, it is sufficient to observe that, for all $t>0,n\in\N$,
    $$
    (I_n(t)+S_n(t))' = -\beta_n I_n(t)(R_n(t)+I_n(t))\leq 0,
    $$
    hence $S_n+I_n$ is non-increasing with respect to the $t$ variable. Because it is non-negative, it converges to a finite limit. Because $S_n$ itself converges, $I_n$ also converges to a finite limit. 

    We also know that $u_n(t) = \int_0^t I_n(\tau) d\tau$ converges to a finite limit owing to Proposition \ref{prop cv}, hence $I_n(t)$ must converge to $0$. This concludes the proof.
\end{proof}

We now turn to the proof of Proposition \ref{prop cv}. The key point is to prove that the functions $(u_n)_{n\in\N}$ are supersolutions and subsolutions of simpler equations.
\begin{corollary}
The functions $(u_n)$ satisfy the following differential inequalities
        \begin{equation}\label{sursol}
                u_n'(t) \geq S^0_n f\left(\sum_i \alpha_{i,n}u_i(t)\right)\left(1 - \beta_n   u_n(t)\right)   - \beta_n I_n^0 u_n(t) + I_n^0,\quad t>0,\ n\in \N,
        \end{equation}
    and
        \begin{equation}\label{subsol}
                 u_n'(t) \leq S^0_nf\left(\sum_i \alpha_{i,n}u_i(t)\right) + S_n^0  \beta_n\left( \frac{1 - e^{-\alpha_{n,n}u_n}}{\alpha_{n,n}} \right) - \beta_n S_n^0 u_n  - \beta_n I_n^0 u_n(t) + I_n^0,\quad t>0, \ n\in \N.
        \end{equation}
\end{corollary}

\begin{proof}
    {\em Step 1. Proof of \eqref{sursol}.}\\
Because $t\mapsto u_k(t)$ is non-decreasing with respect to $t$ for all $k\in\N$, we have
$$
\int_0^t u_n'(\tau) f\left(\sum_k \alpha_{k,n}u_k(\tau) \right)d\tau\leq \left(\int_0^t u_n'(\tau)d\tau\right) f\left(\sum_k \alpha_{k,n}u_k(t) \right)= u_n(t)f\left(\sum_k \alpha_{k,n}u_k(t) \right).
$$
Therefore, we get, from \eqref{eq u loc}
    $$
u_n'(t) \geq S_n^0f\left(\sum_k \alpha_{k,n}u_k(t)\right) - \beta_n S_n^0 u_n(t)f\left(\sum_k \alpha_{k,n}u_k(t) \right) - u_n(t) \beta_n I_n^0 + I_n^0,
$$
hence \eqref{sursol} follows.

\medskip
{\em Step 2. Proof of \eqref{subsol}}.\\
Because $u_k(\tau)\geq 0$ for all $\tau\geq 0, k\in \N$, we have
$$
f\left(\sum_k \alpha_{k,n}u_k(\tau) \right)\geq f\left( \alpha_{n,n}u_n(\tau) \right).
$$
Hence, denoting $F(x)=x -(1-e^{-x})$ the primitive of $f$ vanishing at $x=0$, we get
$$
\int_0^t u_n'(\tau) f\left(\sum_k \alpha_{k,n}u_k(\tau) \right)d\tau\geq \int_0^t u_n'(\tau) f\left( \alpha_{n,n}u_n(\tau) \right)d\tau=\frac{1}{\alpha_{n,n}}F(\alpha_{n,n}u_n(t)).
$$
Using this in \eqref{eq u loc} yields
$$
    u_n'(t) \leq S_n^0f\left(\sum_k \alpha_{k,n}u_k(\tau)\right) - \frac{\beta_n S_n^0}{\alpha_{n,n}} F(\alpha_{n,n}u_n(t)) - u_n(t) \beta_n I_n^0 + I_n^0,  
$$
which rewrites as \eqref{subsol}.
\end{proof}

Using all that, we can turn to the proof of Proposition \ref{prop cv}.

\begin{proof}[Proof of Proposition \ref{prop cv}]

Owing to Proposition \ref{prop u}, we have that, for each $n$, the function $u_n(t)$ converges toward a finite limit when $t$ goes to $+\infty$. We denote $U_n$ its limit.

\medskip
{\em Step 1. Lower bound}.

We start with the lower bound, that is, we now prove that
\begin{equation*}
U_n\geq \frac{1}{\beta_n}.
\end{equation*}

We have $u_n'(t) = I_n(t)  \underset{t\to+\infty}{\longrightarrow} 0$, hence, taking the limit $t\to+\infty$ in \eqref{sursol} yields
        $$
   0 \geq S^0_n f\left(\sum_i \alpha_{i,n}U_i\right)\left(1 - \beta_n   U_n\right)   - \beta_n I_n^0 U_n + I_n^0,\quad \forall n \in \N,
    $$
    that is,
    \begin{equation*}
    \left(S^0_n f\left(\sum_i \alpha_{i,n}U_i\right)+I_n^0\right)(\beta_n U_n -1)\geq 0,\quad \forall n \in \N.
    \end{equation*}
    To conclude, it is sufficient to say that, for all $n\in\N$, we have $S^0_n f\left(\sum_i \alpha_{i,n}U_i\right)+I_n^0>0$. This is where the fact that the graph $G$ is connected (as required in the hypotheses in Section \ref{sec hyp}) comes into play. 

    Indeed, if $n\in \N$ is such that $I_n^0>0$, then $S^0_n f\left(\sum_i \alpha_{i,n}U_i\right)+I_n^0 >0$, and this yields that $U_n \geq \frac{1}{\beta_n}$. Now, if $U_n>0$ for some $n\in \N$, and if $k$ is any neighbor node, we have $S^0_k f\left(\sum_i \alpha_{i,k}U_i\right) >0$ because $\alpha_{n,k}\neq 0$. Iterating this argument, and using that the graph $G$ is connected, this yields the result.

\medskip
{\em Step 2. Upper bound}.

Taking the limit $t\to+\infty$ in \eqref{subsol}, we get that, for each $n$ (using that $f\leq 1$)
$$
\beta_n S_n^0 U_n \leq S_n^0+\frac{S_n^0\beta_n}{\alpha_{n,n}}+ I_n^0(1-\beta_n U_n).
$$
Therefore, because $U_n\geq \frac{1}{\beta_n}$ owing to the first step, this yields that
$$
U_n \leq \frac{1}{\beta_n} + \frac{1}{\alpha_{n,n}},
$$
hence the result.
\end{proof}

\subsection{Speed of propagation}

This section is dedicated to the proof of Theorem \ref{th speed}.

We start with a technical result, a {\em comparison principle} for linear discrete equations.

\begin{prop}[Comparison principle]\label{comp}

Let $(a_n(t))_{n\in\N},(b_n(t))_{n\in\N}$ be two families of $C^1$ functions, such that $(a_n)$(resp. $(b_n)$) is bounded from above (resp. below) uniformly with respect to $n\in\N$ and such that
$$
a_n'(t)\leq S_n^0\sum_k \alpha_{k,n}a_k(t),\quad t>0, n\in \N,
$$
and 
$$
b_n'(t)\geq S_n^0\sum_k \alpha_{k,n}b_k(t),\quad t>0, n\in \N.
$$
Assume in addition that $a_n(0)\leq b_n(0)$ for all $n\in\N$.

Then
$$
a_n(t)\leq b_n(t),\quad \forall t>0, n\in \N.
$$
\end{prop}
This result is a {\em linear comparison principle}, it is similar to the {\em parabolic comparison principle} (see \cite{protter2012maximum}). We want to emphasize that the hypotheses on the parameters $\alpha_{k,n}$ are crucial here - the result might be false without. 

\begin{proof}[Proof of Proposition \ref{comp}.]
    Define $c_n(t) =  b_n(t)-a_n(t)$. These functions are bounded from below uniformly with respect to $n\in\N$, they satisfy $c_n(0)\geq 0$ and
    $$
    c_n'(t) \geq S_n^0\sum_k \alpha_{k,n}c_k(t),\quad t>0,\ n\in \N.
    $$
    Let us take $\lambda>0$ large enough so that
    \begin{equation}\label{hyp l 2}
    \lambda>2\sup_{n\in\N}S_n^0\sum_k \alpha_{k,n}.
    \end{equation}
    Owing to the hypotheses from Section \ref{sec hyp}, this is possible.

    Take $\e>0$ and define
    $$
    c_n^\e(t)=c_n(t)+\e e^{\lambda t}(d(n)+1),\quad t>0,\ n\in \N.
    $$
    We have, for all $n\in\N$ and $t>0$,
    $$
    (c^\e_n)'(t)\geq S_n^0 \sum_k \alpha_{k,n} c_k^\e(t)  +\e \lambda e^{\lambda t}(d(n)+1) - \e e^{\lambda t}S_n^0\sum_k \alpha_{k,n}(d(k)+1).
    $$
    Owing to the hypotheses, we have that, for all $k$ such that $\alpha_{k,n}\neq 0$, $d(k)\leq d(n)+1$. Hence, because \eqref{hyp l 2} yields $\lambda (d(n)+1)\geq S_n^0\sum_k \alpha_{k,n}(d(n)+2)$, we get
    \begin{equation}\label{eq c e}
    (c_n^\e)'(t)\geq S_n^0 \sum_k \alpha_{k,n}c_k^\e(t),\quad \forall t>0,\ n\in \N.
    \end{equation}
    Let us define $T = \sup\{t>0 \ : \  c_n^\e(\tau)\geq 0\quad \forall \tau \in [0,t],\ \forall n\in \N \}$. We have $T>0$ because $c_n^\e(0)\geq \e$. Assume that $T<+\infty$: we will see that this leads to a contradiction.\\

    For all $\eta>0$, we can find $n_\eta\in\N$ such that $c_{n_\eta}^\e(T+\eta)<0$. Remembering that $c_n$ is bounded from below independently of $t,n$, we can find $K>0$ such that $c_n(t)\geq - K$ for all $t>0,n\in \N$. Hence,
    $$
    0>c_{n_\eta}^\e(T+\eta) \geq -K +\e e^{\lambda(T+\eta)}(d(n_\eta)+1).
    $$
    Therefore, when $\eta$ goes to zero, the sequence $(n_\eta)_{\eta>0}$ stays bounded. Up to extraction, it converges to a limit $n_0\in \N$, and we have
    $$
    c_{n_0}^\e(T)=0 \ \text{and}\ c_{n_0}^\e(t)\geq 0,\quad  \forall t\leq T.
    $$
    Then, we have $(c_{n_0}^\e)'(T) \leq 0$, hence \eqref{eq c e} gives
    $$
    \sum_k\alpha_{k,n}c_k^\e(T)\leq 0.
    $$
    Hence, $c_k^\e(T)=0$ for all $k$ such that $\alpha_{k,n}\neq0$, that is, for all $k$ neighbors of $n_0$. Iterating this argument, we get that $c_n^\e(T)=0$ for all $n\in\N$. This is not possible: by definition of $c_n^\e(T)$, it goes to $+\infty$ when $d(n)$ goes to $+\infty$.\\

    Hence, we have reached a contradiction: we have $T=+\infty$. Therefore, $c_n^\e(t)\geq 0$ for all $t>0,\ n\in \N$, and then
    $$
    c_n(t)\geq -\e e^{\lambda t}(d(n)+1),\quad \forall t>0,\ n\in \N.
    $$
    Because this is true for all $\e>0$, the result follows.
\end{proof}
With the above comparison principle at hand, we can give upper bounds on the spreading speed. 
\begin{prop}[Upper bound on the speed]\label{prop up speed} Let $(S_n,I_n)$ be the solution of \eqref{syst loc} with initial datum $(S_n^0 , I_n^0)$ that satisfy the hypotheses of Section \ref{sec hyp}.

Then, we have
    $$
    \sup_{d(n)\geq ct}I_n(t) \underset{t\to+\infty}{\longrightarrow} 0 \ \text{and}\ 
\sup_{d(n)\geq c t}\vert S_n(t) - S_n^0\vert \underset{t\to +\infty}{\longrightarrow}0,\quad
\forall c>c^\star.
    $$
\end{prop}

\begin{proof}[Proof of Proposition \ref{prop up speed}]
Because $S_n(t)$ is non-increasing with respect to $t$, we have $S_n(t)\leq S_n^0$. Using that in \eqref{syst loc}, we find
$$
I_n'(t) \leq  S_n^0 \sum_k \alpha_{k,n}I_k(t),\quad t>0,\ n\in \N.
$$
Let $c>c^\star$ be chosen, and take $c' \in ]c^\star,c[$. By the definition of $c^\star$, we can find $\lambda>0$ such that 
$$
 S_n^0\sum_k \alpha_{k,n} e^{\lambda(d(n)- d(k))}\leq \lambda c',\quad \forall n\in \N.
$$
For such a choice of $\lambda,c'>0$, and for $A>0$ to be chosen large enough after, we define, for all $t>0$ and $n\in \N$,
$$
v_n(t) = A\exp(-\lambda(d(n)-c't)).
$$
Observe that
$$
v_n'(t) - S_n^0\sum_k\alpha_{k,n}v_k(t) = \left(\lambda c'  - S_n^0\sum_k \alpha_{k,n} e^{\lambda(d(n)- d(k))}\right)v_n \geq 0,
$$
Hence, up to choosing $A$ large enough so that $v_n(0)\geq I_n(0)$ (which is possible owing to the hypotheses on $I_n^0$ in Section \ref{sec hyp}), the comparison principle, Proposition \ref{comp}, applied to $I_n,v_n$ yields that
\begin{equation}\label{est I speed}
I_n(t)\leq v_n(t),\quad \forall t>0,\ n\in \N.
\end{equation}
Therefore, 
$$
\sup_{d(n)\geq c t}I_n(t) \leq Ae^{-\lambda(c-c')t} \underset{t\to+\infty}{\longrightarrow}0.
$$
Observe that \eqref{est I speed} also yields that
$$
u_n(t) = \int_0^t I_n(\tau)d\tau \leq \frac{A}{\lambda c'} e^{-\lambda(d(n) - c't)}.
$$
Now, remembering \eqref{Su}, we have $S_n(t) = S_n^0 e^{-\sum_k \alpha_{k,n}u_k(t) }$. Using that $\alpha_{k,n}\neq 0$ implies that $\vert d(k)-d(n)\vert \leq 1$, this yields
$$
\begin{array}{rl}
S_n(t) &\geq S_n^0\exp\left( -\sum_k \frac{\alpha_{k,n}A}{\lambda c'}e^{-\lambda(d(k) - c't)}  \right)\\
&\geq S_n^0\exp\left( - e^{\lambda}\left(\sum_k \frac{\alpha_{k,n}A}{\lambda c'}\right)  e^{-\lambda(d(n) - c't)}  \right).
\end{array}
$$
On the other hand, since $S_n(t)\leq S_n^0$ for all $t>0,n\in\N$, this yields that
$$
\sup_{d(n)\geq c t}\vert S_n(t) - S_n^0\vert \underset{t\to +\infty}{\longrightarrow}0.
$$
\end{proof}
Proposition \ref{prop up speed} proves the first part of Theorem \ref{th speed}, that is, the opinion does not spread faster than $c^\star$.\\

We now turn to the proof of the lower bound on the speed. We start with the following intermediate result.

\begin{lemma}\label{lem speed}
Let $(S_n,I_n)$ be the solution of \eqref{syst loc} with initial datum $(S_n^0 , I_n^0)$ that satisfy the hypotheses of Section \ref{sec hyp}, and let $(u_n)$ be defined by \eqref{def u}.\\

    Let $\eta \in ]0,1[$, and define
    \begin{equation}\label{c eta}
    c_\eta = \inf_{\substack{ n,m\in\N^2 \\ d(n,m)=1}}\frac{S_m^0 (1-\eta)\beta_m f(\alpha_{n,m}\frac{\eta}{\beta_n})}{\eta}>0.
    \end{equation}
    For $c<c_\eta$, we have
    $$
    \liminf_{t\to+\infty}\left(\inf_{d(n)\leq c t} \beta_n u_n(t)\right) \geq \eta.
    $$
\end{lemma}
\begin{proof}[Proof of Lemma \ref{lem speed}]
Let $\eta \in ]0,1[$ be fixed. For each $n$, we have that $t\mapsto u_n(t)$ is non decreasing and converges to a limit $U_n$ such that $U_n\geq \frac{1}{\beta_n}$, owing to Theorem \ref{th cv}. We define, for each $n$,
$$
\tau_n = \inf\{t\geq 0 \ : \ \beta_n u_n(t)\geq \eta\}.
$$
The first step of the proof consists in giving an estimate on $\tau_n$.\\

\medskip
{\em Step 1. Estimate on $\tau_n$.}

Let $n\in \N$ be fixed and take $m\in \N$, $m\neq n$, be such that $\alpha_{n,m}>0$. Then we have $d(n,m)=1$. Assume that $\tau_m > \tau_n$.

Owing to \eqref{sursol}, we have, for $t\in [\tau_n,\tau_m]$,
$$
u'_m(t) \geq S_m^0 f(\alpha_{n,m}u_n)(1-\beta_m u_m(t)),\quad \forall t\in [\tau_n,\tau_m],
$$
hence, because $f$ is non-decreasing,
$$
u'_m(t)\geq S_m^0f(\alpha_{n,m}\frac{\eta}{\beta_n})(1-\eta),\quad \forall t\in [\tau_n,\tau_m],
$$
and then, integrating with respect to $t$ for $t\in [\tau_n,\tau_m]$, we get
$$
\frac{\eta}{\beta_m}\geq (\tau_m-\tau_n)S_m^0f(\alpha_{n,m}\frac{\eta}{\beta_n})(1-\eta).
$$
Therefore, we find that
$$
\tau_m\leq \tau_n + \frac{\eta}{S_m^0(1-\eta)\beta_m f(\alpha_{n,m}\frac{\eta}{\beta_n})},
$$
hence, 
\begin{equation}\label{tm tn}
\tau_m \leq \tau_n + \frac{1}{c_\eta}.
\end{equation}
We supposed that $\tau_m>\tau_n$ to prove this, but this relation is also trivially true when $\tau_m\leq \tau_n$. Therefore, \eqref{tm tn} is true for all $n,m$ such that $d(n,m)=1$.\\

This implies that, for all $n\in\N$,
$$
\tau_n \leq \tau_0 + \frac{d(n)}{c_\eta}.
$$
\medskip
{\em Step 2. Lower bound on the propagation.}

Let $c\in [0,c_\eta[$ be chosen. We want to show that
$$
\liminf_{t\to+\infty}\left(\inf_{d(n) \leq c t} \beta_n u_n(t)\right)\geq \eta.
$$
To do so, let us choose a sequence $(t_k)_{k\in\N}>0$ such that $t_k\underset{k\to+\infty}{\longrightarrow} +\infty$, and let us choose $n_k$ such that $d(n_k)\leq c t_k$ and such that
$$
\beta_{n_k}u_{n_k}(t_{k}) = \inf_{d(n)\leq c t_k} \beta_n u_n(t_k).
$$
We want to show that $\liminf_{k\to+\infty}\beta_{n_k}u_{n_k}(t_{k}) \geq \eta$. 
%Let us argue by contradiction, assume that there is $\e>0$ such that $\e<\frac{\eta}{\beta}$ and we have (up to extraction) $\lim_{k\to+\infty}u_{n_k}(t_k) = \e$.

If we had $d(n_k)$ bounded, the result would follow directly from Theorem \ref{th cv}. Assume then that $d(n_k)\to +\infty$. Then, for $k$ large enough, we have
$$
\tau_{n_k} \leq t_0 + \frac{d(n_k)}{c_\eta} \leq \frac{d(n_k)}{c}\leq t_k.
$$
For such $k$, we have, using that the functions $u_n(t)$ are non-decreasing with respect to the $t$ variable,
$$
\beta_{n_k}u_{n_k}(t_k)\geq  \beta_{n_k}u_{n_k}(\tau_{n_k}) \geq \eta.
$$
This yields the result.
\end{proof}
We can now use this to get the following result, that gives us a lower bound on the speed of propagation.
\begin{prop}[Lower bound on the speed]\label{prop low speed}
Let $(S_n,I_n)$ be the solution of \eqref{syst loc} with initial datum $(S_n^0 , I_n^0)$ that satisfy the hypotheses of Section \ref{sec hyp}, and let $(u_n)$ be defined by \eqref{def u}. We have
    $$
\liminf_{t\to+\infty}\left(\inf_{d(n)\leq c t} \beta_n u_n(t)\right)\geq 1, \quad \forall c \in [0,c_\star[.
$$
\end{prop}
\begin{proof}

{\em Step 1. Rewriting the problem with translated functions.}

Let us take $c\in [0,c_\star[$. Because $c_\eta$ (defined by \eqref{c eta}) is such that $c_\eta \underset{\eta \to 0}{\longrightarrow} c_\star $, we take $\eta$ small enough so that $c_\eta>c$.\\

    We take a sequence of times $(t_k)_{k\in\N}$ that go to $+\infty$ as $k$ goes to $+\infty$ and a sequence $(n_k)_{k\in \N}$ such that $d(n_k)\leq c t_k$ and
    $$
    \inf_{d(n)\leq c t_k} \beta_n u_n(t_k) = \beta_{n_k}u_{n_k}(t_k).
    $$
    Let us show that $\liminf_{k\to+\infty} \beta_{n_k }u_{n_k}(t_k)\geq 1$.\\
    
    If we had $d(n_k)$ bounded independently of $k$, then the result would follow as a consequence of Theorem \ref{th cv}. From now on, we assume that $d(n_k)\underset{k\to+\infty}{\longrightarrow}+\infty$.

    We define the sequence of translated functions
    $$
    v_k(t) :=u_{n_k}(t+t_k),\quad \forall t>-t_k.
    $$
Owing to \eqref{sursol}, it satisfies, for $k$ large enough (because $I_{n_k}^0 = 0$ for $k$ large)
    \begin{equation}\label{eq vk}
    v'_k(t)\geq S_{n_k}^0 f\left(\sum_i \alpha_{i,n_k} u_i(t+t_k)\right)(1-\beta_{n_k} v_k(t)), \quad \forall t> -t_k.
    \end{equation}
Because we have $\ul \beta \leq \beta_{n_k}\leq \ol \beta$, we can assume that, up to extraction, we have $\beta_{n_k} \underset{k\to+\infty}{\longrightarrow} \beta_\infty$, for some $\beta_\infty \in [\ul \beta,\ol \beta]$.
    
Our goal in the rest of the proof is to show that $v_k$ converges to a limit function $v_\infty$ when $k$ goes to $+\infty$, and that $\beta_\infty v_\infty$ is always larger than $1$, which will yield the result.

\medskip
{\em Step 2. convergence of the functions $v_k$.}

    Observe that $v_k,v'_k,v_k^{''}$ are bounded independently of $k$. Indeed, owing to Proposition \ref{prop cv}, we have $0\leq v_k(t)\leq \frac{1}{\beta_{n_k}} + \frac{1}{\alpha_{n_k,n_k}}\leq \frac{1}{\ul \beta} + \frac{1}{\ul \alpha}$, that is, the sequence $(v_k)_{k\in\N}$ is uniformly bounded (that is, independently of $k$).

    In addition, we have $0\leq v_k'(t)=I_{n_k}(t+t_k) \leq P_{n_k} \leq  \ol \sigma + \ol \iota$, that is, the sequence $(v_k')_{k\in\N}$ is also uniformly bounded (hence the sequence $(v_k)_{k\in\N}$ is uniformly Lipschitz).

    We also have that $\vert v_k''(t)\vert =\vert I_{n_k}'(t+t_k)\vert$ is uniformly bounded.\\

Therefore, it follows from the Arzel\`a-Ascoli that there is $v_\infty\in C^1(\R)$ such that, up to extraction,
    $$
    v_k\underset{k\to+\infty}{\longrightarrow} v_\infty
    $$
    and this convergence is locally $C^1$ in $t\in \R$. \\

 Observe that, because the functions $u_{n_k}$ are non-decreasing with respect to $t$, the function $v_\infty$ is also non-decreasing with respect to $t$.\\
 
 By contradiction, we assume from now on that there is $T\in \R$ such that
    \begin{equation}\label{cont v infty}
        v_\infty(T)<\frac{1}{\beta_\infty}.
    \end{equation}
    Then, for all $t\leq T$, we have $v_\infty(t)<\frac{1}{\beta_\infty}$.\\

    On the other hand, it follows from Lemma \ref{lem speed} that, for all $t\in \R$,
        $$
    \liminf_{k\to+\infty}\sum_i \alpha_{i,n_k}u_i(t+t_k)= 
    \liminf_{k\to+\infty}\sum_i \frac{\alpha_{i,n_k}}{\beta_i} \beta_i u_i(t+t_k)\geq \frac{\ul \alpha}{\ol \beta}  V \eta. 
    $$
%    $$
%    \liminf_{k\to+\infty}\sum_i \alpha_{i,n_k}u_i(t+t_k)\geq \ul \alpha V\frac{\eta}{\beta}.
%    $$
    Indeed, when $\alpha_{i,n_k}\neq 0$, we have $d(i)\leq d(n_k)+1\leq ct_k + 1 \leq c_\eta t_k$ for $k$ large. The same argument also yields 
    \begin{equation*}
    v_\infty(t) \geq \frac{\eta}{\beta_\infty},\quad \forall t \in \R.
    \end{equation*}
    Therefore, taking the limit $k\to+\infty$ in \eqref{eq vk} (up to extraction), we get that $v_\infty$ satisfies
    \begin{equation}\label{eq v infty}
    v'_\infty(t) \geq \ul \sigma f\left(\frac{\ul \alpha}{\ol \beta} V \eta \right)(1-\beta_\infty v_\infty(t)),\quad \forall t\leq T.
    \end{equation}

    \medskip
    {\em Step 3. Conclusion.}

    Let us denote $\gamma :=\ul \sigma f\left(\frac{\ul \alpha}{\ol \beta} V \eta \right)>0$. Integrating \eqref{eq v infty}, we have, for all $z\leq T$,
    $$
    v_\infty(T)e^{\gamma \beta_\infty T} - v_\infty(z)e^{\gamma \beta_\infty z} \geq \frac{1}{\beta_\infty}\left(e^{\gamma \beta_\infty T} - e^{\gamma \beta_\infty z} \right).
    $$
Taking the limit $z\to-\infty$ yields that
$$
v_\infty(T)\geq \frac{1}{\beta_\infty},
$$
which is in contradiction with \eqref{cont v infty}. \\

Therefore, we found that $v_\infty(T)\geq \frac{1}{\beta_\infty}$ for all $T\in \R$, in particular for $T=0$. Hence,
$$
1\leq \beta_\infty v_\infty(0) = \liminf_{k\to+\infty}\beta_{n_k}v_k(0)=\liminf_{k\to+\infty}\beta_{n_k} u_{n_k}(t_k),
$$
which yields the result.
\end{proof}

Now, the conclusion of Theorem \ref{th speed} for $c<c_\star$ follows directly from Proposition \ref{prop low speed}. Indeed, we have
$$
S_n(t) - S_n^0e^{-\sum_i \frac{\alpha_{i,n}}{\beta_i}} = S_n^0(e^{-\sum_i \alpha_{i,n}u_i(t)} - e^{-\sum_i \frac{\alpha_{i,n}}{\beta_i}})
$$
and then, for any $c<c_\star$, Proposition \ref{prop low speed} yields that
$$
\limsup_{t\to+\infty}\sup_{d(n)\leq ct} (S_n(t) - S_n^0 e^{-\sum_i \frac{\alpha_{i,n}}{\beta_i}}) \leq 0
$$

Combining this with Proposition \ref{prop up speed} gives Theorem \ref{th speed}.

\section{The effects of non-local stifling}\label{sec non loc}

\subsection{Propagation}

We now consider the situation where the stifling is non-local, that is, when the parameters $(\beta_{i,n})$ do not satisfy \eqref{beta local} anymore. In this case, the spreaders located at node $n$ take into account the presence of the opinion not only in their own community, but also in the neighboring communities to estimate whether or not they should spread the opinion.\\

We first prove Theorem \ref{th spread non loc}, which states that the opinion still spreads in this situation. Again, it will be obtained by working with the equation \eqref{eq u}. Then, we present some simulations that suggest that the stifling can induce patterns when it is ``sufficiently'' non-local.

\begin{prop}\label{prop non loc}
Let $(S_n,I_n)$ be the solution of \eqref{syst} with initial datum $(S_n^0,I_n^0)$ satisfying the hypotheses of Section \ref{sec hyp}. Let $(u_n)$ be defined by \eqref{def u}.\\

Let $U_n = \lim_{t\to+\infty}u_n(t)$. Then,
$$
\limsup_{d(n)\to +\infty} U_n>0.
$$
\end{prop}

\begin{proof} The convergence of the functions $u_n$ is already given by Proposition \ref{prop u}.

    We shall argue by contradiction. We assume that
    \begin{equation}\label{cont stif}
    \limsup_{d(n)\to +\infty} U_n=0.
    \end{equation}
    Because $U_n\geq 0$, this means that we have $U_n\underset{n\to+\infty}{\longrightarrow}0$. Owing to Lemma \ref{lemma un}, we have, using that $u_n(t)$ is non-decreasing with respect to the $t$ variable and that $f$ is also non-decreasing, for all $n\in\N, t>0$,
    $$
u_n'(t) \geq S_n^0f\left(\sum_k \alpha_{k,n}u_k(t)\right) - u_n(t)\left(\sum_i \beta_{i,n}S_i^0 f\left(\sum_k \alpha_{k,i}u_k(t)\right) \right) - u_n(t)\left(\sum_i \beta_{i,n}I_i^0\right) + I_n^0.
$$
Using again that $f$ is non-decreasing, that $\sum_k \alpha_{k,n}u_k \geq \ul \alpha u_n$ and that $\sum_k \alpha_{k,i}u_i \leq \sum_k \alpha_{k,i}U_i$, we get
    $$
u_n'(t) \geq \ul \sigma f\left( \ul \alpha u_n(t)\right) - u_n(t)\left(\sum_i \beta_{i,n}S_i^0 f\left(\sum_k \alpha_{k,i}U_k\right) \right) - u_n(t)\left(\sum_i \beta_{i,n}I_i^0\right) + I_n^0.
$$
Now, let $\e\in ]0,1[$ be fixed, and let $\eta>0$ be such that $f(x)\geq (1-\eta)x$ for $x\in [0,\e]$.

Owing to \eqref{cont stif}, we have, for $d(n)$ large enough, $\ul \alpha u_n(t)\leq \e$, hence, for such $n$, we have
    $$
u_n'(t) \geq \ul \sigma (1-\eta) \ul \alpha u_n(t) - u_n(t)\left(\sum_i \beta_{i,n}S_i^0 f\left(\sum_k \alpha_{k,i}U_k\right) \right) - u_n(t)\left(\sum_i \beta_{i,n}I_i^0\right) + I_n^0.
$$
Now, observe that, for each $n\in \N$, we have (using $f(x)\leq x$)
$$
\sum_i \beta_{i,n}S_i^0 f\left(\sum_k \alpha_{k,i}U_k\right) \leq \sum_i \beta_{i,n}S_i^0 \sum_k \alpha_{k,i}U_k,
$$
and because $\beta_{i,n}\neq 0$ only if $d(i,n)\leq R_\beta$ and $\alpha_{k,i}\neq 0$ only when $d(k,i)\leq1$, we get that (denoting $K = \max_{n} \#\{i \ : \ d(i,n)\leq R_\beta\}$)
$$
\sum_i \beta_{i,n}S_i^0 f\left(\sum_k \alpha_{k,i}U_k\right) \leq \ol \beta \ol \alpha \ol \sigma K V\left(\max_{k \ :\ d(k,n) \leq R_\beta + 1} U_k\right).
$$
For $d(n)$ large enough, owing to \eqref{cont stif}, we have $\ol \beta \ol \alpha \ol \sigma K V \left(\max_{k \ :\ d(k,n) \leq R_\beta + 1} U_k\right)\leq \frac{\ul \sigma (1-\eta)\ul \alpha}{2}$. 

Then for $d(n)$ large enough, we get (using also that $\left(\sum_i \beta_{i,n}I_i^0\right)+I_n^0=0$ for $d(n)$ large enough)
$$
u_n'(t) \geq \frac{\ul \sigma(1-\eta)\ul \alpha}{2} u_n(t),
$$
Hence, for such $n$ we would get $u_n(t)\to+\infty$ when $t\to+\infty$, which is a contradiction with the boundedness of $u_n$ (given in Proposition \ref{prop u}).
\end{proof}

\begin{remark}
    Observe that we used crucially the hypothesis that $\beta_{i,n}=0$ when $d(i,n)$ is large in the proof. Without this hypothesis, we believe that Theorem \ref{th spread non loc} could be false, in the sense that this could prevent the opinion from spreading through the network.
\end{remark}

\begin{proof}[Proof of Theorem \ref{th spread non loc}]
    The convergences of $I_n,S_n$ work similarly as in the proof of Theorem \ref{th cv}, using Proposition \ref{prop u} (which works for solutions of the system with non-local stifling \eqref{syst}).\\

    The fact that $\liminf_{d(n)\to+\infty} S_n^\infty<S_n^0$ is also a direct consequence of Proposition \ref{prop non loc} with \eqref{Su}, which states that $S_n^\infty = S_n^0e^{-\sum_k \alpha_{k,n}U_k}$.
\end{proof}

\subsection{Emergence of patterns}

We now present some simulations that suggest that the non-local stifling can induce patterns, that is, it can make the equilibria of the system switch from a homogeneous state (the opinion is adopted everywhere with the same rate) to an heterogeneous state, where some nodes adopt the opinion more than others.\\

We consider in our simulations the case where the graph $G$ is simply a path graph (linear graph) with $500$ points, that is, the nodes are $0,1,\ldots,499$ and the node $i$ is connected to $i+1$ and $i-1$ (except for $i=0$ and $i=499$ which are connected only to $i=1$ and $i=498$). Remember that, by convention, the nodes are also connected to themselves in our model.\\

We consider the system with non-local stifling, under the homogeneity hypotheses of Remark \ref{conj hom}. We take 
$$
S^0 = 1, \quad a = 4,\quad b = 10
$$
and
$$
 I_n^0 = 0 \ \text{for }  n\neq0\ \text{and}\ I_0^0 = 0.001 .
$$
We take three different values for $R_\beta$: $R_\beta =0$ (local stifling), $R_\beta= 6$ and $R_\beta = 12$.\\

The simulations show $6$ snapshots of the evolution, where we see the populations of ignorants and spreaders. The ignorants form a front while the spreaders form a pulse.

In the first two simulations, the number of ignorants converge to a homogeneous state, where it is almost equal to $0.8$ at least far away from the node $n=0$. In the second simulation, where $R_\beta=6$, there are some oscillations near the node $n=0$, but they are damped.

 In the third simulation (with $R_\beta = 12$), the oscillations do not disappear, a pattern appears.

\begin{figure}[H]
\begin{center}
    \includegraphics[scale=0.4]{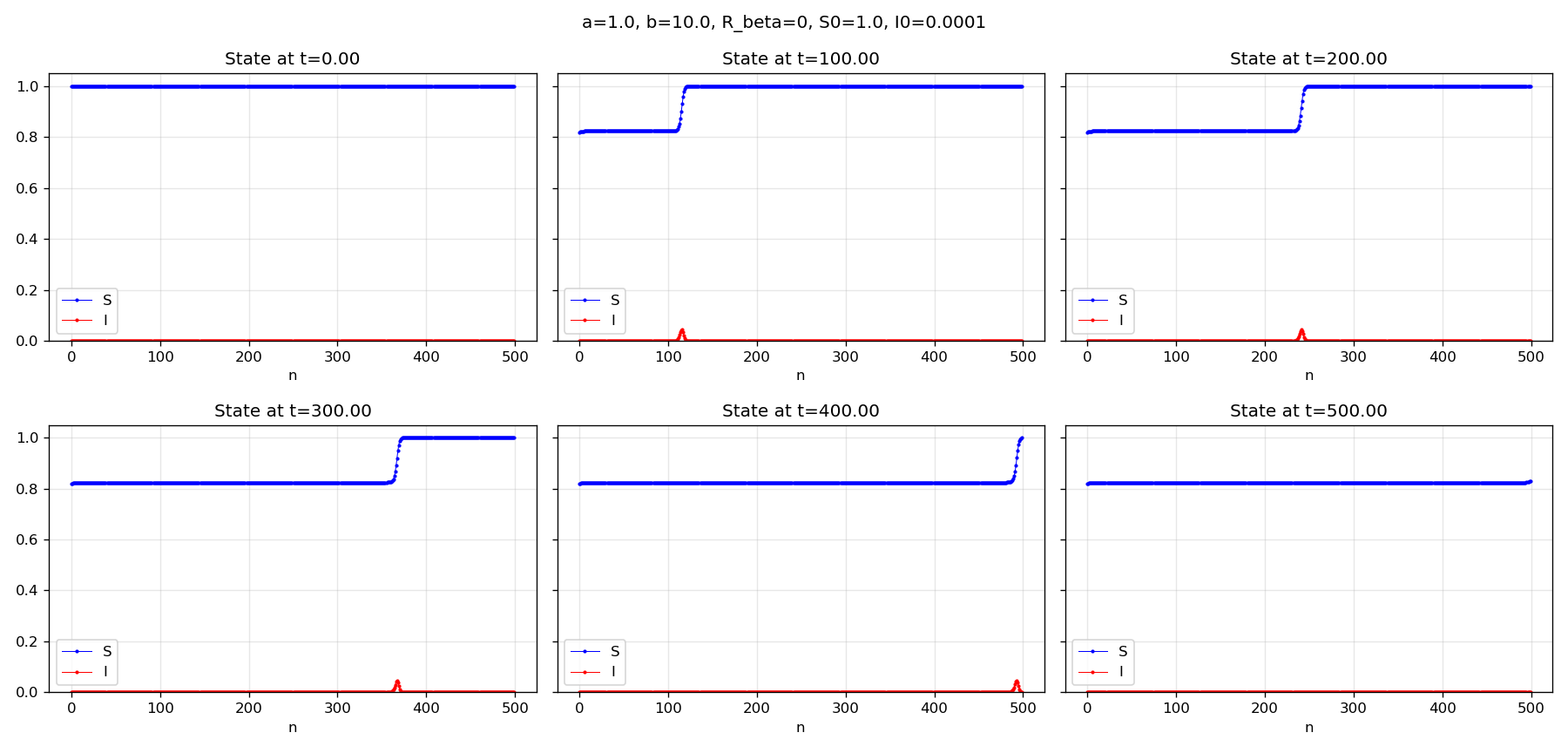}
    \caption{Local stifling, no patterns.}
\end{center}
\end{figure}

\begin{figure}[H]
\begin{center}
    \includegraphics[scale=0.4]{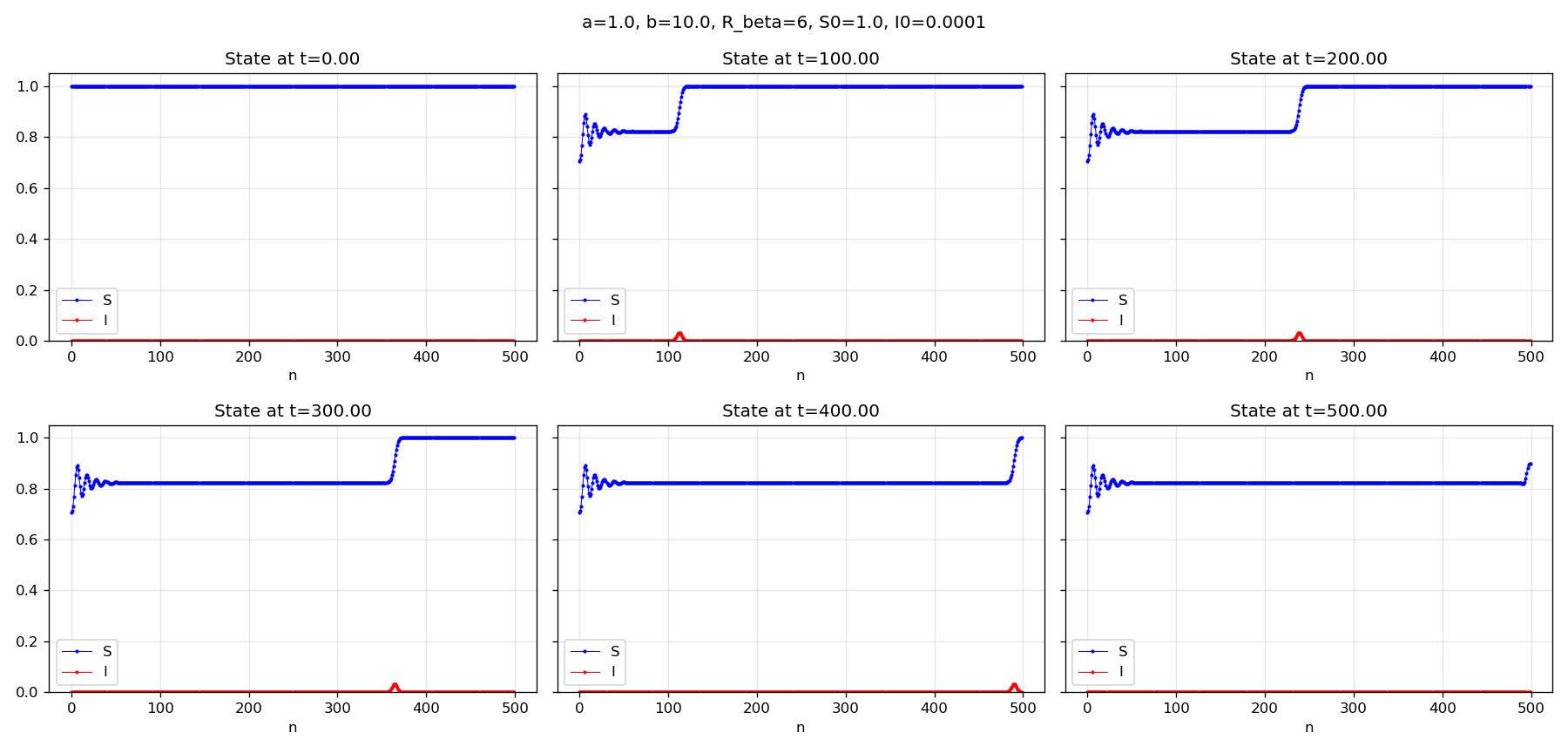}
    \caption{Non-local stifling, no patterns.}
\end{center}
\end{figure}

\begin{figure}[H]
\begin{center}
    \includegraphics[scale=0.4]{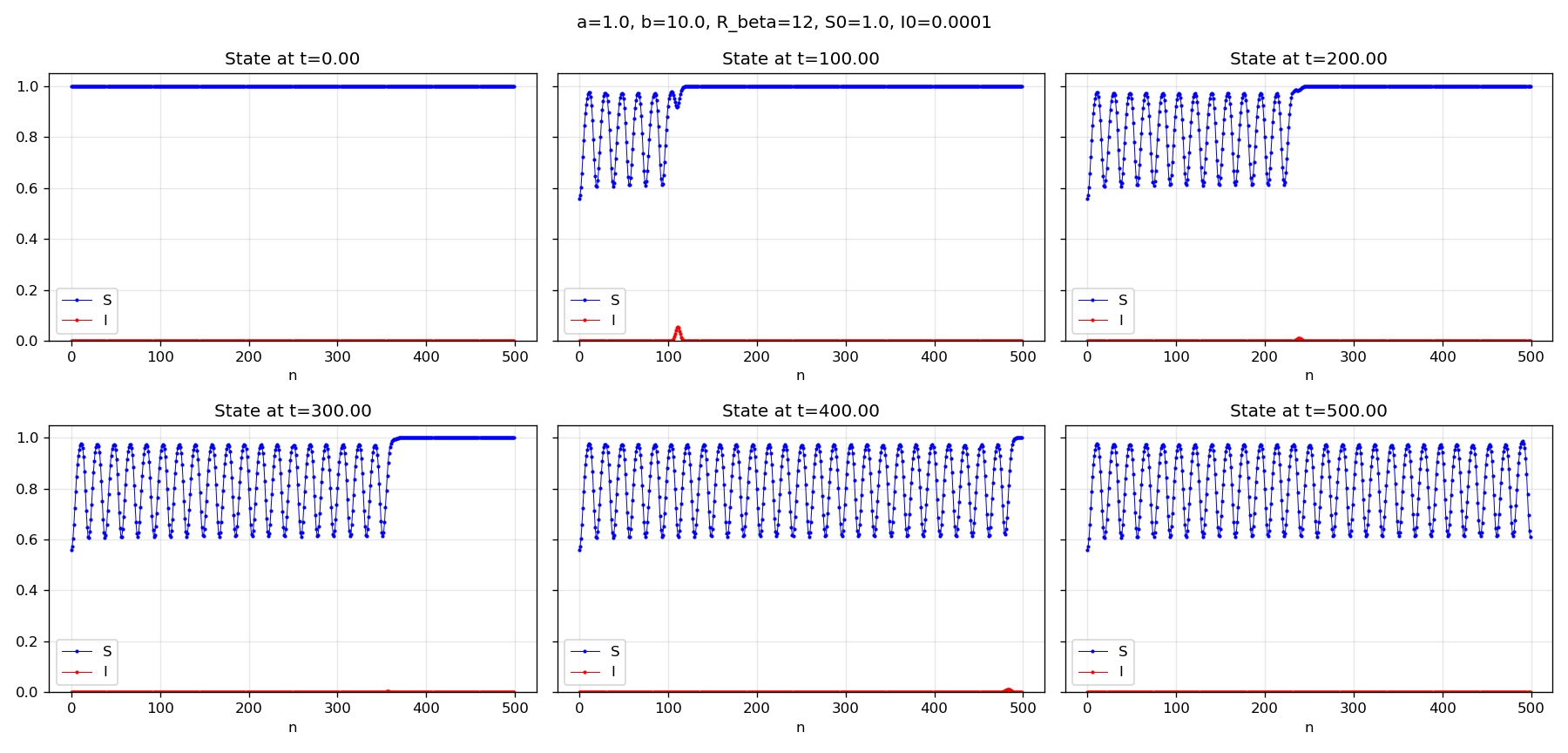}
    \caption{Non-local stifling, emergence of patterns.}
\end{center}
\end{figure}

Let us try to give an heuristic explanation of this phenomenon. As observed in the simulations, the spreaders propagate from left to right by forming a pulse. Suppose that, at a given time $T>0$, there is a rightmost node where the opinion is adopted, let us call it $n$. At this node, we would have $S_n(T)<S_n^0-\eta$ and $I_n(T)>\eta$ say. Let $i>n$ be such that $1 \ll d(i,n)<R_\beta$. The node $i$ is ``far away'' from $n$, so the opinion can not reach node $i$ soon: we can guarantee that $S_i(t)\approx S_i^0$ for some times after time $T$, but the stifling at node $i$ already feels the influence of node $n$ at time $T$. The stifling term at node $i$, that is $\sum_{k} \beta_{k,i}(I_k+R_k)$, could be large even if the opinion has not reached the node $i$. This would block the propagation at node $i$, and at the nodes close to $i$. However, as Theorem \ref{th spread non loc} shows, the opinion can not be blocked, at least not on zones that are ``too large''. This would lead to the formation of alternating zones where the opinion is adopted and others where it is ignored.\\

This is reminiscent of the phenomenon of Turing patterns. Turing patterns appear in systems with two species, one which is an inhibitor and one which is an activator. The activator creates other activators and inhibitors, while the inhibitors hinder the production of activators. Then, if the activator diffuses more slowly than the inhibitor, the system could give rise to patterns. Such phenomenon can be seen in chemical reactions, and is widely used to explain the emergence of patterns in cellular dynamics for instance. We refer to the book \cite{murray2011mathematical} and to the original paper of Turing, \cite{turing1990chemical}. The phenomenon here is more complex however: the mechanism behind the apparition of Turing patterns is the fact that diffusion can make homogeneous states unstable. Here, a linear analysis shows that homogeneous states are always stable. The patterns are created by the propagation itself.
\\

The fact that long-range stifling leads to patterns in our model is of course not restricted to path graphs. For instance, let us show the same simulation as the one presented in Figure $1$ in the introduction, the propagation of the opinion on a Watts-Strogatz graph, but with non-local stifling. We still observe that the final state is not homogeneous.

\begin{figure}[H]
\centering

% -------- Row 1 --------
\begin{minipage}{0.48\textwidth}
\centering
\includegraphics[width=\linewidth]{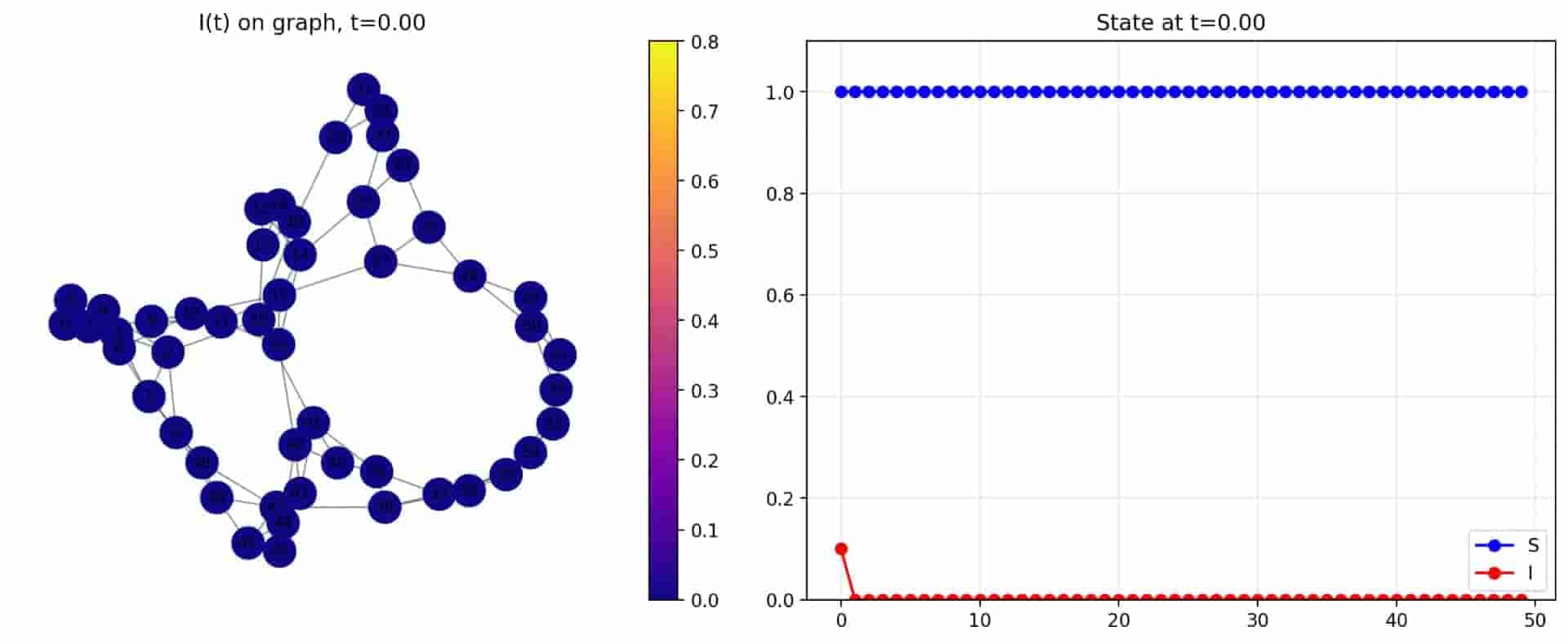}
\end{minipage}
\hfill
\begin{minipage}{0.48\textwidth}
\centering
\includegraphics[width=\linewidth]{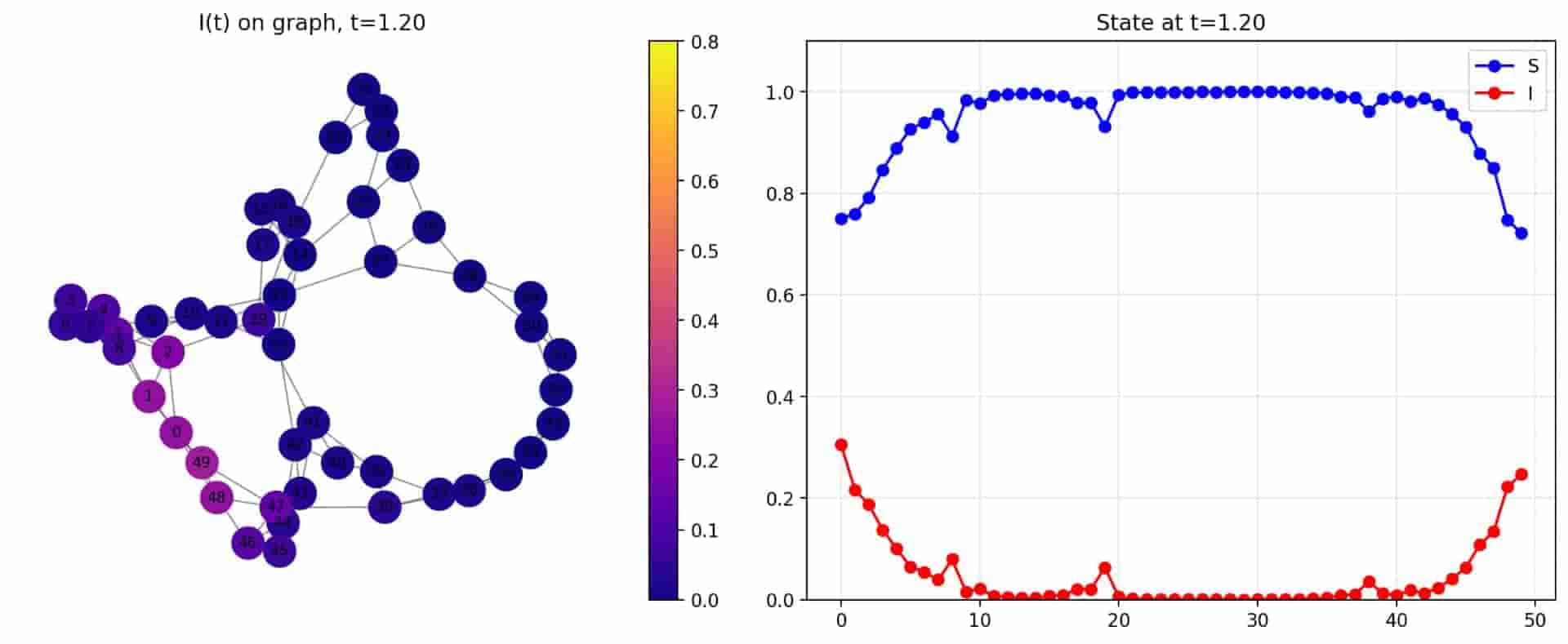}
\end{minipage}

\vspace{0.4cm}

% -------- Row 2 --------
\begin{minipage}{0.48\textwidth}
\centering
\includegraphics[width=\linewidth]{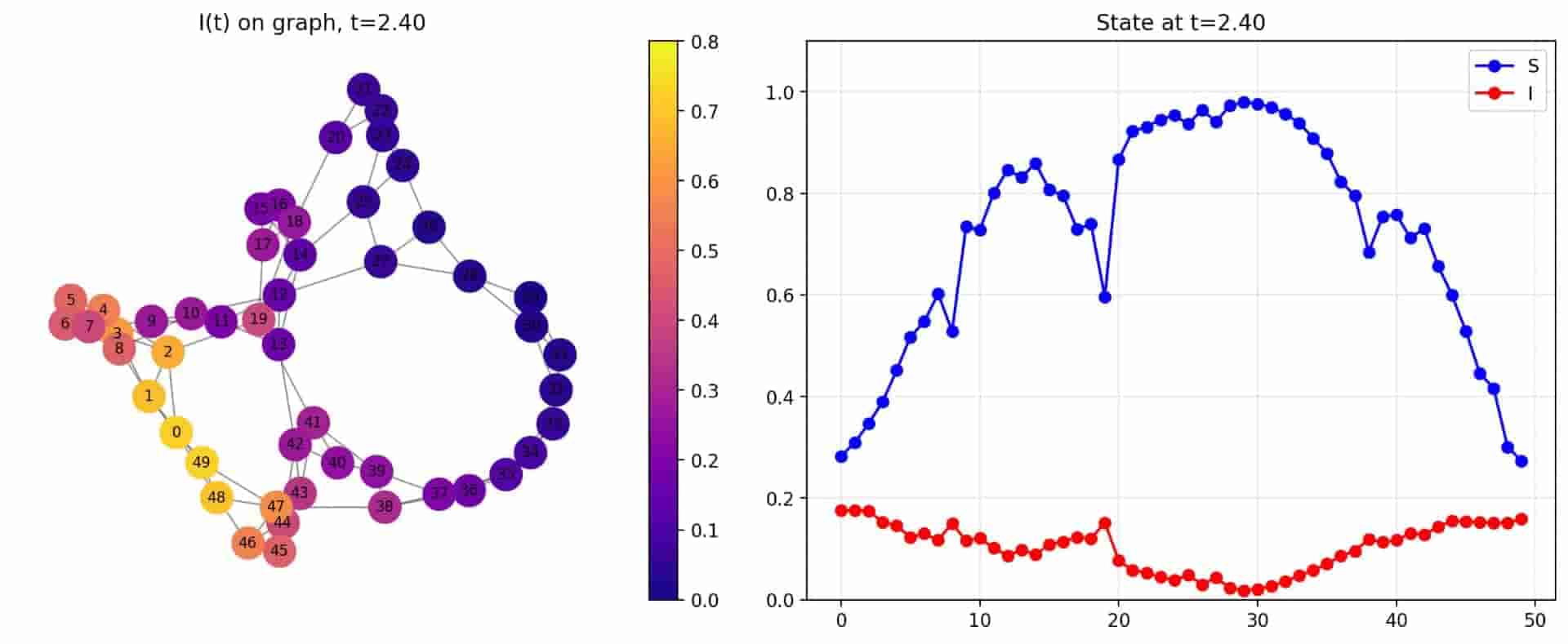}
\end{minipage}
\hfill
\begin{minipage}{0.48\textwidth}
\centering
\includegraphics[width=\linewidth]{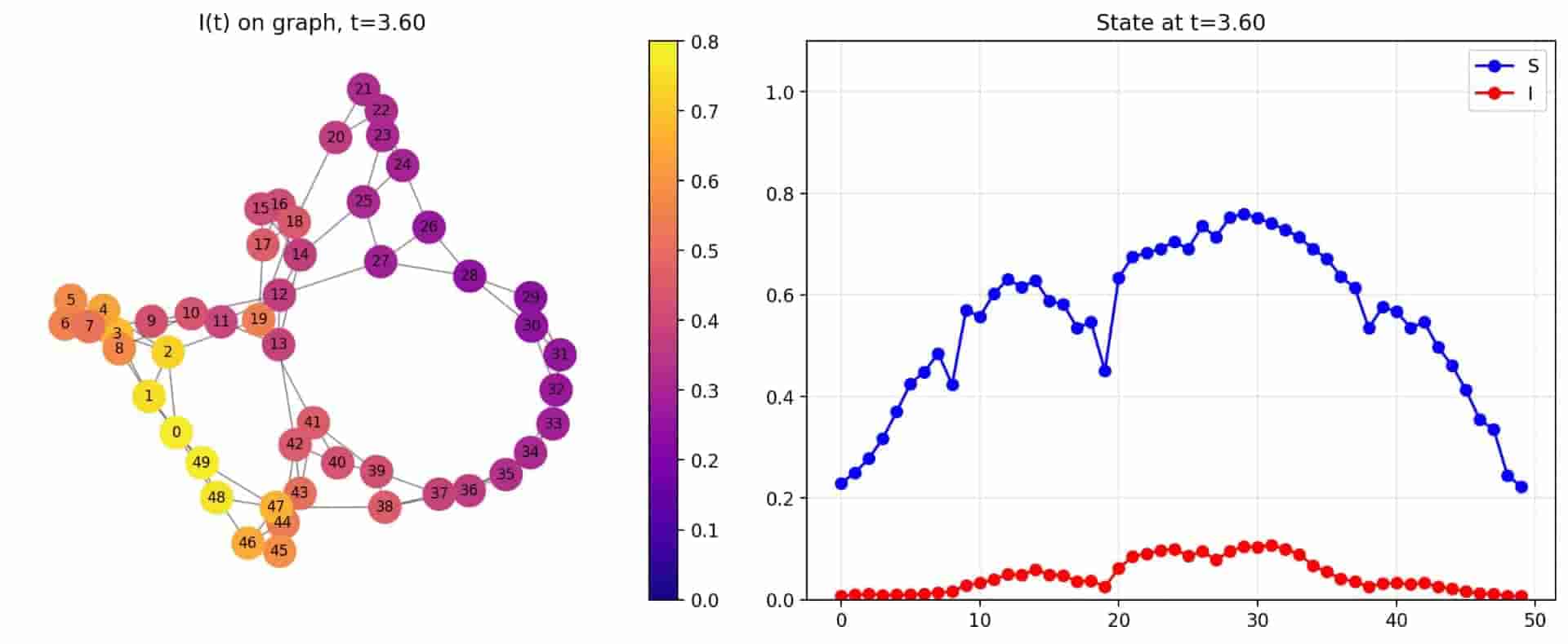}
\end{minipage}

\vspace{0.4cm}

% -------- Row 3 (centered last image) --------
\begin{minipage}{0.48\textwidth}
\centering
\includegraphics[width=\linewidth]{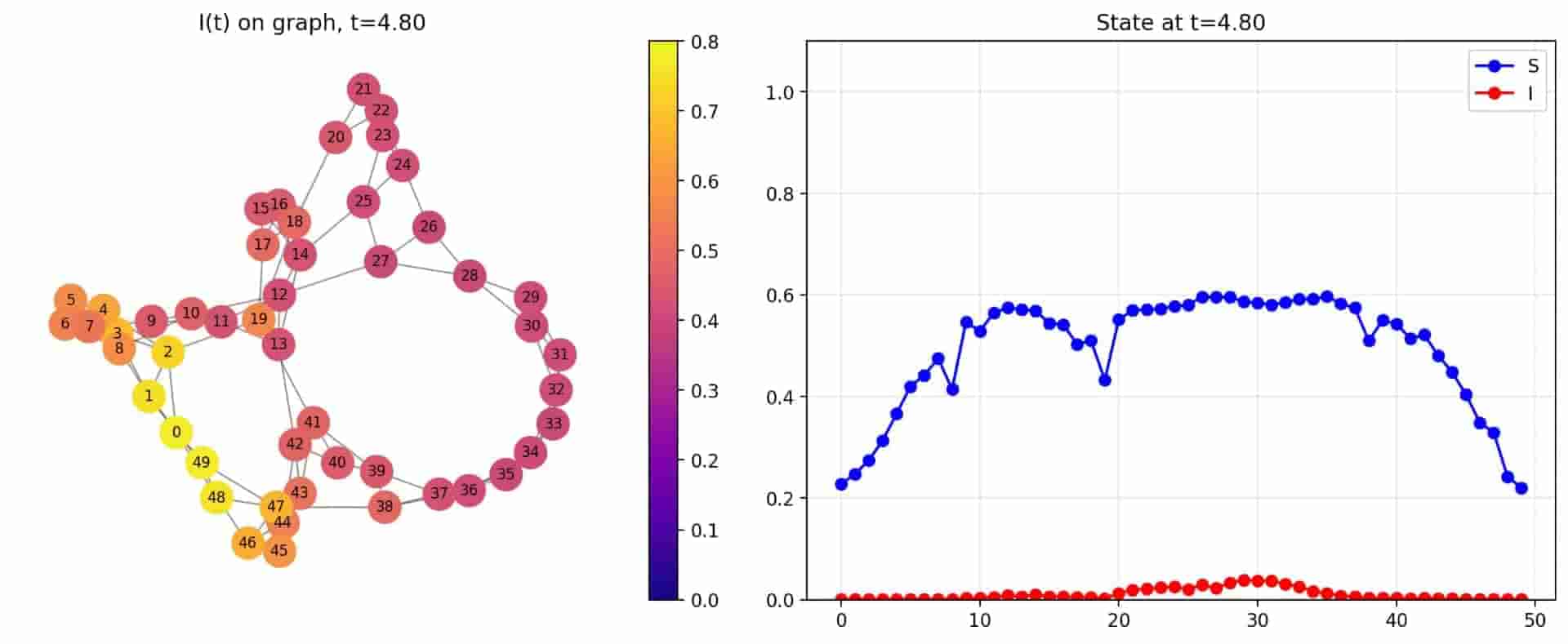}
\end{minipage}
\hfill
\begin{minipage}{0.48\textwidth}
\centering
\includegraphics[width=\linewidth]{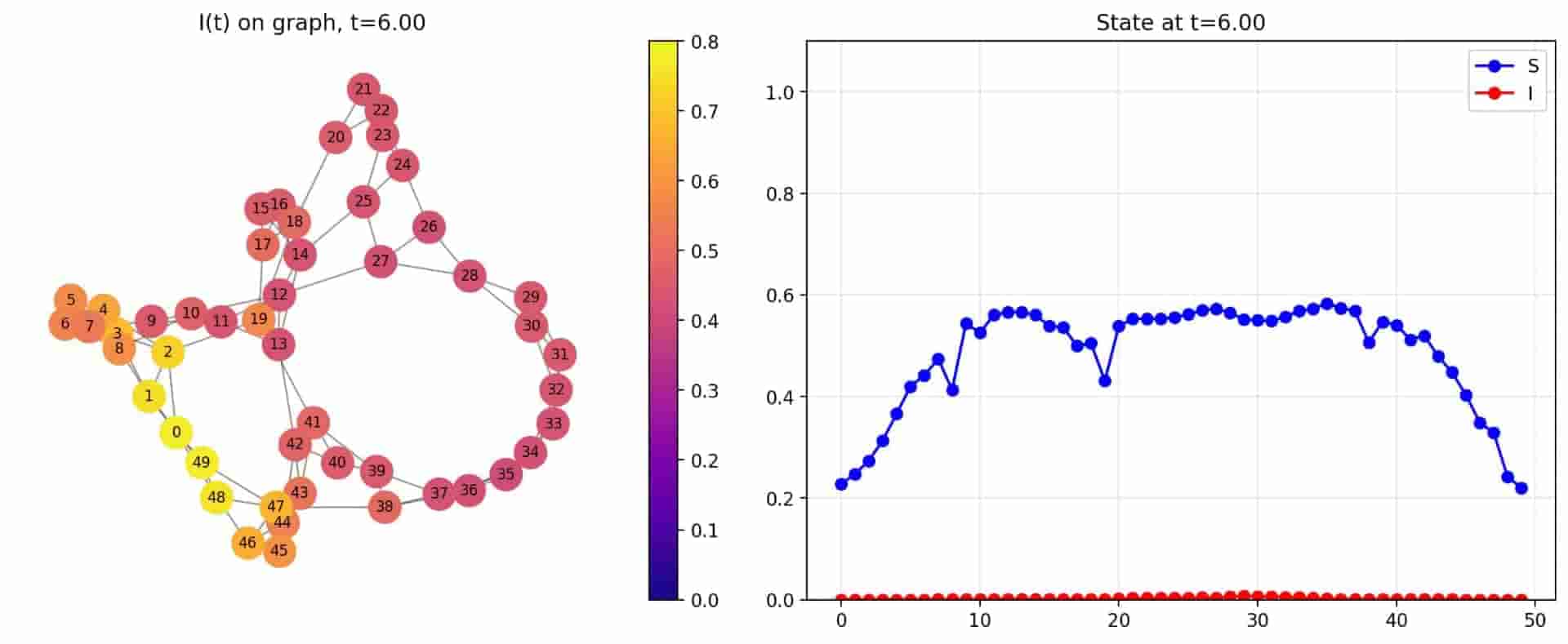}
\end{minipage}

\caption{Propagation of the opinion on a Watts-Strogatz graph. Non-local stifling.}
\end{figure}

\textbf{Acknowledgments.} This study contributes to the IdEx Université de Paris ANR-18-IDEX-0001. The research leading to these results has received funding from the ANR project “ReaCh” (ANR-23-CE40-0023-01).

\bibliographystyle{siam}
\bibliography{biblio}

\end{document}